\documentclass[a4paper,11pt,twoside]{article}
\usepackage{a4wide}
\usepackage{amssymb, amsmath, amsthm,txfonts}
\usepackage{t1enc,lmodern}
\usepackage[utf8]{inputenc}
\usepackage{enumitem}
\usepackage[round]{natbib}
\usepackage[english]{babel}

\usepackage{algorithm}
\usepackage{algorithmic}
\usepackage{hyperref}

\def\cite{\citet}

\emergencystretch=\hsize
\def\ce{{\mathcal E}}
\def\ca{{\mathcal A}}
\def\cf{{\mathcal F}}
\def\cg{{\mathcal G}}
\def\cn{{\mathcal N}}
\def\cv{{\mathcal V}}
\def\E{{\mathbb E}}
\def\N{{\mathbb N}}
\def\P{{\mathbb P}}
\def\R{{\mathbb R}}
\def\Rp{{{\mathbb R}_+^*}}
\def\s{\star}
\def\t{\theta}
\def\vt{\vartheta}

\def\vl{\lambdaup}
\def\l{\lambda}
\def\um{\underline{m}}
\def\ind#1{{\bf 1}_{\left\{#1\right\}}}

\def\abs#1{\left|#1\right|}
\def\Cov{\mathop{\rm Cov}\nolimits}

\def\norm#1{\mathop{\left| #1 \right|}\nolimits}
\def\norm1#1{\mathop{\left| #1 \right|_1}\nolimits}
\def\inv#1{\mathop{\frac{1}{ #1}}\nolimits}
\def\expp#1{\mathop {\mathrm{e}^{ #1}}}
\def\diag{\mathop {\mathrm{diag}}\nolimits}
\def\defeq{\stackrel{\rm def}{=}}
\theoremstyle{plain}
\newtheorem{thm}{Theorem}[section]
\newtheorem{prop}[thm]{Proposition}
\newtheorem{lemma}[thm]{Lemma}
\newtheorem{algo}[thm]{Algorithm}
\newtheorem{rem}[thm]{Remark}
\newtheorem{cor}[thm]{Corollary}
\theoremstyle{definition}

\numberwithin{equation}{section}
\makeatletter
\newcounter{hypo}
\newcommand*{\dohypo}{\textbf{(${\mathcal A}$\thehypo)}}
\newenvironment{hypo}[1][]{%

\refstepcounter{hypo}
\list{}{%
\settowidth{\labelwidth}{\dohypo}%
\setlength{\labelsep}{10pt}%
\setlength{\leftmargin}{\labelwidth}
\advance\leftmargin\labelsep%
}%
\item[\dohypo  #1]%
  }{%
\endlist
}

\def\hypref#1{\hyperref[#1]{(${\mathcal A}$\ref*{#1})}}
\def\hypreff#1#2{\hyperref[#2]{(\mbox{${\mathcal A}$\ref*{#1}-{\it\ref*{#2})}}}}

\makeatother

\newcommand*\samethanks[1][\value{footnote}]{\footnotemark[#1]}

\title{Importance sampling for jump processes and applications to finance}
\date{\today}
\author{Laetitia Badouraly Kassim\thanks{Univ. Grenoble Alpes, Laboratoire Jean
Kuntzmann, BP 53, 38041 Grenoble Cédex 9, FRANCE, e-mail:
Laetitia.Badouraly-Kassim@ensimag.imag.fr,  jerome.lelong@imag.fr,
imane.loumrhari@ensimag.imag.fr.  \newline This project was supported by the
Finance for Energy Market Research Centre, www.fime-lab.org.} \and Jérôme
Lelong\samethanks[1] \and Imane Loumrhari\samethanks[1]}

\begin{document}
\maketitle

\begin{abstract}
  Adaptive importance sampling techniques are widely known for the Gaussian
  setting of Brownian driven diffusions. In this work, we want to extend them to
  jump processes. Our approach relies on a change of the jump intensity combined
  with the standard exponential tilting for the Brownian motion. The free
  parameters of our framework are optimized using sample average approximation
  techniques. We illustrate the efficiency of our method on the valuation of
  financial derivatives in several jump models. \\

\end{abstract}

\section{Introduction}

Lévy models have become quite popular in finance over the last decade. Vanilla
options are easily and efficiently priced using the Fast Fourier Transform
approach developed by \cite{CarrMadan99} but things become far more delicate
for exotic options, for which Monte Carlo often reveals as the only possible
approach from a numerical point of view. This becomes even more true when
dealing with high dimensional products. In this work, we want to propose an
adaptive Monte Carlo method based on importance sampling for computing the
expectation of a function of a Lévy process. As explained
by \cite{Kiessling:2011uq}, when resorting to Monte
Carlo approaches, infinite activity Lévy processes are often approximated by
finite activity processes, which can always be represented as the sum of
a continuous diffusion (ie. driven by a Brownian motion) and a compound Poisson
process. In this work, we will concentrate on such jump diffusions with a
Brownian driven part and a jump part written as a compound Poisson process or
possibly the sum of independent compound Poisson processes in the
multidimensional case. \\

We consider a mixed Gaussian and Poisson framework in which we would like to
settle an adaptive Monte Carlo method based on some importance sampling approach.
Let $G = (G_1, \dots, G_d)$ be a standard normal random vector in $\R^d$ and
$N^\mu = (N^{\mu_1}_1, \dots, N^{\mu_p}_p)$ a vector of $p$ independent
Poisson random variables with parameters $\mu = (\mu_1, \dots,
\mu_p)$. We assume that $G$ and $N^\mu$ are independent. We focus
on the computation of 
\begin{align}
  \label{eq:PE}
  \ce = \E[f(G,N^\mu)]
\end{align}
where $f : \R^d \times \N^p \longrightarrow \R$ satisfies $\E[|f(G,N^\mu)|] <
\infty$.

\begin{lemma}
  \label{lem:chgt}
  For any measurable function $h : \R^d \times \N^p \longrightarrow \R$ either
  nonnegative or such that $\E[|h(G,N^\mu)|] < \infty$, one has
  \begin{align}
    \label{eq:IS}
    \forall \; \t \in \R^d, \l \in \Rp^p, \qquad \E[h(G,N^\mu)] =
    \E\left[ h(G+\t, N^\l) \expp{- \t \cdot G - \frac{|\t|^2}{2}} \prod_{i=1}^p
    \expp{\l_i -\mu_i} \left(\frac{\mu_i}{\l_i}\right)^{N^{\l_i}_i}
    \right]
  \end{align}
  where $N^\l$ is a vector of $p$ independent Poisson random variables with
  parameters $\l = (\l_1, \dots, \l_p)$.
\end{lemma}
The proof of this lemma relies on elementary variable changes.
Lemma~\ref{lem:chgt} enables us to introduce some extra degrees of freedom in
the computation of $\ce$. When the expectation $\ce$ is computed using a Monte
Carlo method, the Central Limit Theorem advises to use the representation of
$f(G,N^\mu)$ with the smallest possible variable which is achieved by choosing
the parameters $(\t, \l)$ which minimize the variance of of $f(G+\t, N^\l)
\expp{- \t \cdot G - \frac{|\t|^2}{2}} \prod_{i=1}^p \expp{\l_i -\mu_i}
\left(\frac{\mu_i}{\l_i}\right)^{N^{\l_i}_i}$. This raises several
questions which are investigated in the paper. Does the variance of $f(G+\t,
N^\l) \expp{- \t \cdot G - \frac{|\t|^2}{2}} \prod_{i=1}^p \expp{\l_i -\mu_i}
\left(\frac{\mu_i}{\l_i}\right)^{N^{\l_i}_i}$ admits a unique minimizer? If so,
how can it be computed numerically and how to make the most of it in view of a
further Monte Carlo computation? \\

These questions are quite natural in the context of Monte Carlo computations and
have already been widely discussed in the pure Gaussian framework. The first
applications to option pricing of some adaptive Monte Carlo methods based on
importance sampling goes back to the papers of \cite{arouna1,arouna2}. These
papers were based on a change of mean for the Gaussian random normal vectors and
the optimal parameter was searched for by using some stochastic approximation
algorithm with random truncations. This approach was later further investigated
by \cite{lap-lelo-11} who proposed a more general framework for settling
adaptive Monte Carlo methods using stochastic approximation, which is know to be
a little tricky to fine tune in practical applications. To circumvent the
delicate behaviour of stochastic approximation, \cite{jour-lelo-09} proposed to
resort to sample average approximation instead, which basically relies on
deterministic optimization techniques. An alternative to random truncations was
studied by \cite{LP10} who managed to modify the initial problem in order ta
apply the more standard Robbins Monro algorithm. Not only have they applied
this to the Gaussian framework but they have considered a few examples of Levy
processes relying on the Esscher transform to introduce a free parameter. The
idea of using the Esscher transform was also extensively investigated by
\cite{kawai07,Kawai-levy08,Kawai-levy09}. \\

In this work, we want to understand how the jump intensity of a Lévy process can
be modified to reduce the variance. First, we explain the parametric importance
sampling transformation we use for the Gaussian and Poisson parts. Then, in
Section~\ref{sec:optimal}, we prove that this transformation leads to a convex
optimization problem and we study the properties of the optimal parameter
estimator. Then, in Section~\ref{sec:adap}, we explain how to use this estimator
in a Monte Carlo method. We prove that this approach satisfies an adaptive
strong law of large numbers and a central limit theorem with optimal limiting
variance.  Finally, in Section~\ref{sec:apps}, we apply our methodology to option
pricing with jump processes. 

\paragraph{Notations.}
\begin{itemize}
  \item We encode any elements of $\R^m$ as column vectors.
  \item If $x \in \R^m$, $x^*$ is a row vector. We use the ``${}^*$''
    notation to denote the transpose operator for vectors and matrices.
  \item If $x, y \in \R^m$, $x \cdot y$ denotes the scalar
    product of $x$ and $y$ and the associated norm is denoted by $|\cdot|$.
  \item If $x \in \R^m$, $\diag_m(x)$ is the matrix with diagonal elements given
    by the vector $x$ and all extra diagonal elements equal to zero.
  \item The matrix $I_m$ denotes the identity matrix in dimension $m$.
  \item If $x \in \R^m$, we defined $d_0(x) = \min_{1 \le i \le m} |x_i|$ which
    is the distance between $x$ and the set $\{ y \in \R^m \; : \;
    \prod_{i=1}^m y_i = 0 \}$.
  \item We say that a random vector $X$ with values in $\R^m$ has Poisson
    distribution with parameter $\mu \in \R^m$ if the $X_i$ are independent
    and have Poisson distribution with parameter $\mu_i$.
\end{itemize}

\section{Computing the optimal importance sampling parameters}
\label{sec:optimal}
\subsection{Properties of the variance}

Thanks to Lemma~\ref{lem:chgt}, the expectation $\ce$ can be written
\begin{align*}
  \ce = \E\left[ f(G+\t, N^\l) \expp{- \t \cdot G - \frac{|\t|^2}{2}}
  \prod_{i=1}^p \expp{\l_i -\mu_i}
  \left(\frac{\mu_i}{\l_i}\right)^{N^{\l_i}_i} \right], \quad \forall
  \; \t \in \R^d, \l \in \Rp^p.
\end{align*}
Note that for the particular choice of $\t=0$ and $\l=\mu$, we recover
Equation~\eqref{eq:PE}.

The convergence rate of a Monte Carlo estimator of $\ce$ based on this new
representation is governed by the variance  of $f(G+\t, N^\l) \expp{- \t \cdot
G - \frac{|\t|^2}{2}} \prod_{i=1}^p \expp{\l_i -\mu_i}
\left(\frac{\mu_i}{\l_i}\right)^{N^{\l_i}_i}$ which can be written in the
form $v(\t, \l) - \ce^2$ where
\begin{align}
  \label{eq:var}
  v(\t, \l) =   \E\left[ f(G, N^\mu)^2 \expp{- \t \cdot G + \frac{|\t|^2}{2}}
  \prod_{i=1}^p \expp{\l_i -\mu_i} \left(\frac{\mu_i}{\l_i}\right)^{N^{\mu_i}_i}
  \right].
\end{align}
This expression of $v$ is easily obtained by applying Lemma~\ref{lem:chgt} to
the function $h(g, n) = f(g+\t, n)^2 \expp{- 2 \t \cdot g - |\t|^2}
\prod_{i=1}^p \expp{2 (\l_i -\mu_i)}
\left(\frac{\mu_i}{\l_i}\right)^{2 n_i}$. Applying the change of measure
backward after computing the variance enables us to write the variance in a form
which does not involve the parameters $\t$ and $\l$ in the arguments of the
function $f$. This remark is of prime importance as it is the basement of the
following key result stating the strong convexity of $v$.
\begin{prop}
  \label{prop:convex}
  Assume that
  \begin{hypo}
    \label{hyp0}
    \begin{subhypo}  
    \item \label{positive} $\exists (n_1, \dots, n_p) \in {\N^*}^p, \; s.t. \;
      \P(|f(G, (n_1, \dots, n_p))| >0 ) > 0$
    \item\label{integrable} $\exists \gamma>0, \quad \E\left[  |f(G,
      N^\mu)|^{2+\gamma}  \right] < \infty$.
    \end{subhypo}
  \end{hypo}
  Then, the function $v$ is infinitely continuously differentiable, strongly
  convex and moreover the gradient vectors are given by
  \begin{align}
    \label{eq:gradv_t}
    \nabla_{\t} v(\t, \l) = & \E\left[ \left(\t - G\right)
    f(G, N^\mu)^2 \expp{- \t \cdot G + \frac{|\t|^2}{2}}
    \prod_{i=1}^p \expp{\l_i -\mu_i} \left(\frac{\mu_i}{\l_i}\right)^{N^{\mu_i}_i}
    \right]
    \\
    \label{eq:gradv_mu}
    \nabla_{\l} v(\t, \l) = & \E\left[ a(N^\mu,\l) f(G, N^\mu)^2 \expp{- \t \cdot G +
    \frac{|\t|^2}{2}} \prod_{i=1}^p \expp{\l_i -\mu_i}
    \left(\frac{\mu_i}{\l_i}\right)^{N^{\mu_i}_i} \right]
  \end{align}
  where the vector $a(N^\mu,\l) = \left(1 - \frac{N_1^{\mu_1}}{\l_1}, \dots, 1 -
  \frac{N_p^{\mu_p}}{\l_p}\right)^*$.
  The second derivatives are defined by
  \begin{align}
    \label{eq:hessv_t}
    \nabla^2_{\t,\t} v(\t, \l) = & \E\left[ \left( I_d + (\t - G) (\t -G)^*\right)
    f(G, N^\mu)^2 \expp{- \t \cdot G + \frac{|\t|^2}{2}} \prod_{i=1}^p \expp{\l_i
    -\mu_i} \left(\frac{\mu_i}{\l_i}\right)^{N^{\mu_i}_i} \right] \\
    \label{eq:hessv_tmu}
    \nabla^2_{\t,\l} v(\t, \l) = & \E\left[ (\t - G) a(N^\mu,\l)^*
    f(G, N^\mu)^2 \expp{- \t \cdot G + \frac{|\t|^2}{2}} 
    \prod_{i=1}^p \expp{\l_i -\mu_i}
    \left(\frac{\mu_i}{\l_i}\right)^{N^{\mu_i}_i} \right] \\
    \label{eq:hessv_mu}
    \nabla^2_{\l,\l} v(\t, \l) = & 
    \E\left[ \left(D+ a(N^\mu,\l) a(N^\mu,\l)^* \right)
    f(G, N^\mu)^2 \expp{- \t \cdot G + \frac{|\t|^2}{2}} 
    \prod_{i=1}^p \expp{\l_i -\mu_i} \left(\frac{\mu_i}{\l_i}\right)^{N^{\mu_i}_i} \right]
  \end{align}
  where the diagonal matrix $D$ is defined by $D =
  \diag_p\left(\frac{N_1^{\mu_1}}{\l_1^2}, \dots,
  \frac{N_p^{\mu_p}}{\l_p^2}\right)$. 
\end{prop}

\begin{proof}
  Let us define the function $F: \R^d \times \R^d \times \Rp^p \longrightarrow
  \R$ by
  \begin{equation}
    \label{eq:F}
    F(g,\t,n,\l) = f(g, n)^2 \expp{- \t \cdot g + \frac{|\t|^2}{2}}
    \prod_{i=1}^p \expp{\l_i -\mu_i} \left(\frac{\mu_i}{\l_i}\right)^{n_i}.
  \end{equation}
  For any values of $(g,n)$, the function $(\t,\l) \longmapsto F(g, \t, n, \l)$
  is infinitely continuously differentiable. Since for all $0 < \um < M$,
  \begin{align}
    \label{eq:boundgradtheta}
    \sup_{|(\t,\l)| \le M, \um < d_0(\l) } |\partial_{\t_j}
    F(G,\t,N^\mu,\l)|& \le  \left(M + \expp{G_j} + \expp{-G_j} \right) f(G,
    N^\mu)^2 \expp{M^2/2 + pM}\nonumber \\
    & \quad \prod_{k=1}^d (\expp{M G_k} + \expp{-M G_k}) \prod_{i=1}^p \expp{-\mu_i}
    \left(\frac{\mu_i}{\um}\right)^{N^{\mu_i}_i} 
  \end{align}
  where the right hand side is integrable because by Hölder's inequality and
  Assumption~\hypreff{hyp0}{integrable}, we have that for all $(\t,\l) \in \R^d
  \times \R^p$,
  $\E(f(G,N^\mu)^2 \expp{\t \cdot G + \l \cdot N^\mu}) < \infty$.
  Hence, Lebesgue's theorem ensures that $v$ is continuously differentiable
  w.r.t. $\t$ and $\nabla_\t v$ is given by Equation~\eqref{eq:gradv_t}.

  We proceed similarly for the derivative w.r.t. $\l$ by using the following
  upper bound 
  \begin{align}
    \label{eq:boundgradmu}
    \sup_{|(\t,\l)| \le M, \um < d_0(\l) } |\partial_{\l_j}
    F(G,\t,N^\mu,\l)| & \le  \left(1 + \expp{N^{\mu_j}_j/\um} +
    \expp{-N^{\mu_j}_j/m} \right) f(G, N^\mu)^2 \expp{M^2/2 + pM}
    \nonumber \\
    & \quad \prod_{k=1}^d (\expp{M G_k} + \expp{-M G_k}) \prod_{i=1}^p \expp{-\mu_i}
    \left(\frac{\mu_i}{\um}\right)^{N^{\mu_i}_i}. 
  \end{align}
  High order differentiability properties are obtained by similar arguments and
  in particular the Hessian matrix writes with the help of the function $F$
  \begin{align*}
    \nabla^2 v(\t, \l)  = & E\left[ F(G,\t,N^\mu, \l) 
    \begin{pmatrix}
      (\t - G) (\t -G)^* &  (\t - G) a(N^\mu,\l)^* \\
      a(N^\mu,\l) (\t - G)^*  & a(N^\mu,\l) a(N^\mu,\l)^*
    \end{pmatrix} \right. \nonumber \\
    & \left. \quad + F(G,\t,N^\mu, \l) 
    \begin{pmatrix}
      I_d & 0 \\
      0 & D
    \end{pmatrix}   \right] 
  \end{align*}
  Note that
  \begin{align*}
    \begin{pmatrix}
      (\t - G) (\t -G)^* &  (\t - G) a(N^\mu,\l)^* \\
      a(N^\mu,\l) (\t - G)^* & a(N^\mu,\l) a(N^\mu,\l)^* 
    \end{pmatrix} = 
    \begin{pmatrix}
      \t - G \\
      a(N^\mu,\l) 
    \end{pmatrix} 
    \begin{pmatrix}
      \t - G \\
      a(N^\mu,\l) 
    \end{pmatrix}^*. 
  \end{align*}
  Hence the first part of the Hessian is a positive semi definite rank one
  matrix.
  \begin{align*}
    \E \left[ F(G,\t,N^\mu, \l) 
    \begin{pmatrix}
      I_d & 0 \\
      0 & D
    \end{pmatrix} \right] & \ge \E [ F(G,\t,N^\mu, \l)
    \ind{N^\mu=(n_1,\dots,n_p}] \diag\left(I_d, \frac{n_1}{\l_1^2}, \dots,
    \frac{n_p}{\l_p^2} \right).
  \end{align*}
  Moreover,
  \begin{align*}
    \E [ F(G,\t,N^\mu, \l) \ind{N^\mu=(n_1,\dots,n_p)}] 
    & \ge \E\left[f(G, (n_1,\dots,n_p))^2 \expp{-\t \cdot G +
    \frac{|\t|^2}{2}}\right] 
    \prod_{i=1}^p \expp{n_i - 2 \mu_i} \left(\frac{\mu_i^2}{n_i}\right)^{n_i}
    \inv{n_i !} \\
    & \ge \E\left[f(G, (n_1,\dots,n_p))^2 \expp{-\t \cdot G} \right] 
    \E\left[\expp{\t \cdot G}\right] 
    \prod_{i=1}^p \expp{n_i - 2 \mu_i} \left(\frac{\mu_i^2}{n_i}\right)^{n_i}
    \inv{n_i !} \\
    & \ge \E\left[|f(G, (n_1,\dots,n_p))| \right]^2
    \prod_{i=1}^p \expp{n_i - 2 \mu_i} \left(\frac{\mu_i^2}{n_i}\right)^{n_i}
    \inv{n_i !}
  \end{align*}
  Thanks to Condition~\hypreff{hyp0}{positive}, this lower bound is strictly positive.
  Hence, the Hessian matrix is uniformly bounded from below which yields the
  strong convexity of $v$.
\end{proof}
As a consequence, the function $v$ admits a unique minimizer $(\t_\s, \l_\s)$
defined by $\nabla_\t  v(\t_\s, \l_\s) = \nabla_\l v(\t_\s, \l_\s)  = 0$. The
characterization of $(\t_\s,\l_\s)$ as the unique minimizer of a strongly convex
function is very appealing but there is no hope to compute the gradient of $v$
in a closed form, so we will need to resort to some kind of approximations
before running the optimization step. Before studying the possible ways of
approximating the optimal parameter, let us note that that it is of dimension
$d+p$ which can become very large in particular when the variables $G$ and $N^l$
come from the discretization of jump diffusion process. In many situations, it
is advisable to reduce the dimension of the space in which the optimization
problem is solved.

\paragraph{Reducing the dimension of the optimization problem.}

Let $0 < d' \le d$ and $0 < p' \le p$ be the reduced dimension. Instead of
searching for the best importance sampling parameter $(\t, \l)$ in the whole
space $\R^d \times \Rp^p$, we consider the subspace $\{ (A \vt, B \vl) \; : \;
\vt \in \R^{d'}, \; \vl \in \Rp^{p'} \}$ where  $A \in \R^{d \times d'}$ is a
matrix with rank $d' \le d$ and $B \in \Rp^{p \times p'}$ a matrix with rank $p'
\le p$. Note that since all the coefficients of $B$ are non negative, for all
$\vt \in \Rp^{p'}$, $B \vt \in \Rp^{p}$; actually, it is easily seen that the
image of $\Rp^{p'}$ through $B$ is isomorphic to $\Rp^{p'}$.  

For such matrices $A$ and $B$, we introduce the function $v^{A,B} : \R^{d'}
\times \Rp^{p'} \longmapsto \R$ defined by  
\begin{align}
  \label{eq:vAB}
  v^{A,B}(\vt, \vl) = v (A \vt, B \vl)
\end{align}
The function $v^{A,B}$ inherits from the regularity and convexity properties of
$v$. Hence, from Proposition~\ref{prop:convex}, we know that $v^{A,B}$ is
continuously infinitely differentiable and strongly convex. As a consequence,
there exists a unique couple of minimizers $(\vt^{A,b}_\s, \vl^{A,B}_\s)$ such
that  $v^{A,B} (\vt_\s^{A,B}, \vl_\s^{A,B}) =  \inf_{\vt \in \R^{d'}, \vl \in
\Rp^{p'}} v^{A,B} (\vt, \vl)$. We can also deduce the gradient vector of
$v^{A,B}_n$
\begin{align*}
  \nabla v^{A,B} (\vt, \vl) &= 
  \begin{pmatrix} 
    A^* \nabla_\t (A \vt, B \vl) \\
    B^* \nabla_\l (A \vt, B \vl) 
  \end{pmatrix}
\end{align*}
and its Hessian matrix
\begin{align*}
  \nabla^2 v^{A,B} (\vt, \vl) &= \E \left [ F(G,A \vt, N^\mu, B \vl) 
  \begin{pmatrix}
    A^* (A \vt - G) (A \vt -G)^* A &  A^*(A \vt - G) a(N^\mu,B \vl)^* B\\
    B^*  a(N^\mu,B \vl) (A \vt - G)^* A & B^* a a^*(N^\mu,B \vl)  B
  \end{pmatrix} \right.  \nonumber \\
  & \left. \quad + F(G,A \vt,N^\mu, B \vl) 
  \begin{pmatrix}
    A^* A & 0 \\
    0 & B^* D B
  \end{pmatrix} \right]
\end{align*}
where the function $F$ is defined by Equation~\eqref{eq:F}. For the particular
choices  $A = I_d$, $B = I_p$, $d=d'$ and $p=p'$, the functions $v^{I_d, I_p}$
and $v$ coincide.

\paragraph{The Esscher transform as a way to reduce the dimension.} Consider a
two dimensional process $(X_t)_{t \le T}$ of the form $X_t = (W_t, \tilde
N^{\tilde \mu}_t)$ where $W$ is a real Brownian motion and $\tilde N^{\tilde
\mu}$ is a Poisson process with intensity ${\tilde \mu}$. The Esscher
transform applied to $X$ yields that for any nonnegative function $h$,
we have the following equality $\forall \; \alpha \in
\R, \tilde \l \in \R_+^*$, 
\begin{align*}
  \E[h((W_t,\tilde N^{\tilde \mu}_t), \; t \le T)] = 
  \E\left[ h((W_t + \alpha t,\tilde N^{\tilde \l}, \; t \le T)) 
  \expp{- \alpha W_T - \frac{|\alpha|^2 T}{2}} \expp{T(\tilde \l -\tilde \mu)}
  \left(\frac{\tilde \mu}{\tilde \l}\right)^{\tilde N^{\tilde \l}_T}  \right]
\end{align*}
Let $0 = t_0 < \dots < t_p = T$ be a time grid of $[0,T]$. If we consider 
the vector $G$ (resp. $N^{\mu}$) as the increments of $W$ (resp. $\tilde
N^{\tilde \mu}$) on the grid, we can recover a particular form of
Equation~\eqref{eq:IS} with $A, B \in \R^{p \time 1}$ given by
$$
A = \left( \sqrt{t_1}, \sqrt{t_2 - t_1}, \dots, \sqrt{t_{p} - t_{p-1}}\right)^*;\quad
B = \left( t_1, t_2 - t_1, \dots, t_{p} - t_{p-1}\right)^*.
$$

\subsection{Tracking the optimal importance sampling parameter}

The optimal importance sampling parameter $(\t^*, \l^*)$ can characterized as
the unique zero of an expectation, which is the typical framework for applying
stochastic approximation. In particular, we could use the algorithm introduced
by \cite{chen86}; we refer to \cite{lelong_as, lel:tcl:11} for a study of the
convergence and asymptotic behaviour of these algorithms. The use of
stochastic approximation for devising adaptive importance sampling method was
deeply investigated in a recent survey by \cite{lap-lelo-11} who highlighted the
difficulties in making those algorithms practically converge. \\

In this work, we adopt a totally different point of view often called \textit{sample
average approximation}, which  basically consists in first replacing expectations by
sample averages and then using deterministic optimization techniques on these
empirical means. This approach was studied in the Gaussian framework
by~\cite{jour-lelo-09} and proved to be very efficient. \\

Let $(G^j)_{j \ge 1}$ be a sequence of $d-$dimensional independent and
identically distributed standard normal random variables. We also introduce
$(N^{\mu, j})_{j\ge 1}$ a sequence of $p-$ dimensional independent and
identically distributed random vector following the law of $N^\mu$, ie. the
components of the vectors are independent and Poisson distributed with parameter
$\mu$.

For $n \ge 1$, we introduce the sample average approximation of the function
$v^{A,B}$ defined by
\begin{align}
  \label{eq:vnAB}
  v_n^{A,B} (\vt, \vl) = \inv{n} \sum_{j=1}^n f(G^j , N^{\mu,j})^2 \expp{- A\vt
  \cdot G^j + \frac{|A\vt|^2}{2}} \prod_{i=1}^p \expp{(B\vl)_i -\mu_i}
  \left(\frac{\mu_i}{(B \vl)_i}\right)^{N^{\mu_i,j}_i}.
\end{align}
For $n$ large enough, $f(G^j, N^{\mu,j}) \ne 0$ for some index $j \in \{1,
\dots,n\}$ and the approximation $v_n^{A,B}$ is also strongly convex and hence
admits a unique minimizer $(\vt_n^{A,B}, \vl_n^{A,B})$ defined by $v_n^{A,B}
(\vt_n^{A,B}, \vl_n^{A,B}) =  \inf_{\vt \in \R^{d'}, \vl \in \Rp^{p'}}
v_n^{A,B} (\vt, \vl)$.

\begin{prop}
  \label{prop:vncvu} Under Assumption~\hypref{hyp0}, the sequence of random
  functions $(v_n^{A,B})_n$ converges a.s.  locally uniformly to the continuous
  function $v^{A,B}$.
\end{prop}
To prove this result, we use the uniform strong law of large numbers recalled
hereafter, see for instance \cite[Lemma A1]{MR1241645}. This result is also a
consequence of the strong law of large numbers in Banach spaces \cite[Corollary
7.10, page 189]{ledouxtalagrand}.
\begin{lemma}\label{lem:ulln}
  Let $(X_i)_{i\geq 1}$ be a sequence of i.i.d. $\R^m$-valued random vectors,
  $E$ an open set of $\R^d$ and $h:E\times\R^m\rightarrow \R$ be a measurable
  function. Assume that
  \begin{itemize}
    \item a.s., $\chi\in E \mapsto h(\chi,X_1)$ is continuous,
    \item for all compact sets $K$ of $\R^d$ such that $K \subset E$,
      $\E\left(\sup_{\chi \in K}|h(\chi,X_1)|\right)<+\infty$.
  \end{itemize}
  Then, a.s. the sequence of random functions $\chi\in K
  \mapsto\frac{1}{n}\sum_{i=1}^n h(\chi,X_i)$ converges locally uniformly to the
  continuous function $\chi\in E \mapsto\E(h(\chi,X_1))$.
\end{lemma}
\begin{proof}[Proof of Proposition~\ref{prop:vncvu}]
  It is sufficient to prove the result for $v_n$ and it will hold for
  $v_n^{A,B}$.
  Let $M > \um > 0$. For all $(\t, \l)$ such that $|(\t,\l)| \le M$ and
  $d_0(\l) > \um$, we have
  \begin{align*}
    f(G, N^{\mu})^2 \expp{- \t
    \cdot G + \frac{|\t|^2}{2}} \prod_{i=1}^p \expp{\l_i -\mu_i}
    \left(\frac{\mu_i}{\l_i}\right)^{N^{\mu_i}_i} &\le 
    f(G, N^{\mu})^2 \prod_{k=1}^d (\expp{- M G_k} + \expp{M G_k})
    \expp{\frac{M^2}{2}} \prod_{i=1}^p \expp{M -\mu_i}
    \left(\frac{\mu_i}{\um}\right)^{N^{\mu_i}_i}.
  \end{align*}
  The r.h.s. is integrable by \hypref{hyp0} and Hölder's inequality; hence, we
  can apply Lemma~\ref{lem:ulln}.
\end{proof}

\begin{prop}
  \label{prop:convsaa} Under Assumption~\hypref{hyp0}, the pair $(\vt^{A,B}_n,
  \vl^{A,B}_n)$ converges a.s. to $(\vt^{A,B}_\s, \vl^{A,B}_\s)$  as $n
  \longrightarrow +\infty$. Moreover, if
  \begin{hypo}
    \label{integrable2} $\exists \delta>0, \quad \E\left[  |f(G,
    N^\mu)|^{4+\delta}  \right] < \infty$,
  \end{hypo}
  $\sqrt{n} \left((\vt^{A,B}_n,
  \vl^{A,B}_n) - (\vt^{A,B}_\s, \vl^{A,B}_\s)\right)$ converges in law
  to the normal distribution $\cn_{d+p}(0, \Gamma)$ where
  \begin{align*}
    \Gamma = \left( \nabla^2 v^{A,B}(\vt^{A,B}_\s, \vl^{A,B}_\s)\right)^{-1}
    \Cov(\nabla F(G, A \vt_\s^{A,B}, N^\mu, B \vl_\s^{A,B})) \left( \nabla^2
    v^{A,B}(\vt^{A,B}_\s, \vl^{A,B}_\s)\right)^{-1} 
  \end{align*}
  with the function $F$ defined by Equation~\eqref{eq:F} and its gradient
  computed w.r.t. the reduced parameters $(\vt, \vl)$.
\end{prop}
Condition~\hypref{integrable2} ensures that the covariance matrix $\Cov(\nabla
F(G, A \vt_\s^{A,B}, N^\mu, B \vl_\s^{A,B}))$ does exist.  The non
singularity of the matrix $\nabla^2 v^{A,B}(\vt^{A,B}_\s, \vl^{A,B}_\s)$ is
guaranteed by the strict convexity of $v$. \\

By combining Propositions~\ref{prop:vncvu} and~\ref{prop:convsaa}, we can state
the following result
\begin{cor}
  Under Assumption~\hypref{hyp0}, $v_n^{A,B}(\vt^{A,B}_n, \vl^{A,B}_n)$ converge
  a.s. to  $v^{A,B}(\vt^{A,B}_\s, \vl^{A,B}_\s)$ as $n \longrightarrow +\infty$.
\end{cor}

\begin{proof}[Proof of Proposition~\ref{prop:convsaa}] 
  Let $\varepsilon>0$.  We define a compact neighbourhood $\cv_{\varepsilon}$ of
  $(\vt_\s, \vl_\s)$
  \begin{equation}
    \label{eq:Veps} \cv_{\varepsilon} \defeq \left\{ (\vt, \vl) \in \R^d \times
    \R^p \; : \; |(\vt, \vl) - (\vt_\s, \vl_\s)| \le \varepsilon \right\}.
  \end{equation}
  In the following, we assume that $\varepsilon$ is small enough, so that
  $\cv_{\varepsilon}$ is included in $\R^d \times \Rp^p $.

  By the strict convexity and the continuity of $v^{A,B}$,
  $$
  \alpha \defeq \inf_{(\vt, \vl) \in
  \cv_{\varepsilon}^c}v^{A,B}(\vt, \vl)-v^{A,B}(\vt_\star^{A,B}, \vl_\star^{A,B})>0.
  $$
  The local uniform convergence of $v_n^{A,B}$ to $v^{A,B}$ ensures that for
  some $n_{\alpha}$ sufficiently large, 
  \begin{equation}
    \label{eq:nalpha}
    \forall n\geq n_\alpha,\;\forall (\vt, \vl) \in \cv_{\varepsilon},
    \;|v_n^{A,B}(\vt, \vl)-v^{A,B}(\vt, \vl)|\leq
    \frac{\alpha}{3}.
  \end{equation}
  For $n\geq n_\alpha$ and $(\vt,\vl) \notin \cv_{\varepsilon}$, we define
  $(\vt^{A,B}_\varepsilon, \vl^{A,B}_\varepsilon) \in \cv_\varepsilon$ and
  writes as the convex combination of $(\vt_\s^{A,B}, \vl_\s^{A,B})$ and
  $(\vt,\vl)$. 
  $$
  (\vt^{A,B}_\varepsilon, \vl^{A,B}_\varepsilon) \defeq \left( \vt_\s^{A,B} +
  \varepsilon \frac{\vt-\vt_\star^{A,B}}{|(\vt-\vt_\s^{A,B}, \vl -
  \vl^{A,B}_\s)|}, \vl_\s^{A,B} +
  \varepsilon \frac{\mu-\vl_\star^{A,B}}{|(\vt-\vt_\s^{A,B}, \vl -
  \vl^{A,B}_\s)|}\right).
  $$
  We deduce, using the convexity of $v_n^{A,B}$ for the first
  inequality and Equation~\eqref{eq:nalpha} for the second one
  \begin{align*}
    v_n^{A,B}(\vt,\vl)-v_n^{A,B}(\vt_\s^{A,B}, \vl^{A,B}_\s)&\geq
    \frac{|(\vt-\vt_\s^{A,B}, \vl - \vl^{A,B}_\s)|}{\varepsilon}\left[v_n^{A,B}
    (\vt^{A,B}_\varepsilon, \vl^{A,B}_\varepsilon) -
    v_n^{A,B}(\vt_\star^{A,B}, \vl^{A,B}_\s)\right]\\
    &\geq \left[v^{A,B} (\vt^{A,B}_\varepsilon, \vl^{A,B}_\varepsilon) -
    v^{A,B}(\vt_\star^{A,B}, \vl^{A,B}_\s)
    -\frac{2\alpha}{3}\right]\geq \frac{\alpha}{3}.
  \end{align*}
  The optimality of $(\vt_n^{A,G}, \vl_n^{A,B})$ yields that
  $v_n^{A,B}(\vt_n^{A,B}, \vl_n^{A,B})\leq v_n^{A,B}(\vt_\star^{A,B},
  \vl_\s^{A,B})$. So, we conclude that $(\vt_n^{A,B}, \vl_n^{A,B}) \in
  \cv_\varepsilon$ for $n\geq n_\alpha$. Therefore, $(\vt_n^{A,B}, \vl_n^{A,B})$
  converges a.s. to $(\vt_\star^{A,B}, \vl_\s^{A,B})$. 

  We have seen in the proof of Proposition~\ref{prop:convex}, that $\E\left[
  \sup_{|(\t, \l)| \le M, \um < d_0(\l)} \nabla F(G,\t,N^\mu, \l) \right] <
  \infty$, see Equation~\eqref{eq:boundgradmu} and~\eqref{eq:boundgradtheta}.
  Similarly, we can show that $\E\left[
  \sup_{|(\t, \l)| \le M, \um < d_0(\l)} \nabla^2 F(G,\t,N^\mu, \l) \right] <
  \infty$. The central limit theorem governing the convergence of the pair
  $(\vt^{A,B}_n, \vl^{A,B}_n)$ to the pair $(\vt^{A,B}_\s, \vl^{A,B}_\s)$ can be
  deduced from~\cite[Theorem~A2]{MR1241645}.
\end{proof}

\section{Adaptive Monte Carlo}
\label{sec:adap}

In this section, we assume to have at hand a sequence of optimal solutions
$(\vt_n^{A,B}, \vl_n^{A,B})$ and want to devise an adaptive Monte Carlo taking
advantage of the knowledge of these parameters through the use of
Equation~\eqref{eq:IS}. In a previous work~\cite{jour-lelo-09} dedicated to the
Gaussian framework, we had used the same samples for approximating $v$ by $v_n$
and after to build a Monte Carlo estimator of $\ce$ involving $\t_n$. This was
possible because a normal random vector $X$ with mean vector $\t$ naturally
writes as $X = \t + G$ where $G$ is a standard normal random vector.

No such simple relation exists for the Poisson distribution to
link a Poisson random variable with parameter $\mu$ to one with parameter
$\l$.  Hence, it is not worth trying to reuse, for the Monte Carlo estimator
based on Equation~\eqref{eq:IS}, the same Poisson random samples as those
involved in $v_n$.  Then, we suggest the following two stages algorithm.
\begin{algo}\label{algo:is}\hfill
  \begin{description}
    \item[First stage] Generate a sequence $(G^j)_{j = 1, \dots, m}$ of i.i.d
      random vector following the standard normal distribution in $\R^{d}$ and a
      sequence $(N^j = (N_1^j, \dots, N_{p}^j))_{j = 1, \dots, m}$ of i.i.d
      Poisson random vectors with parameter $\mu$. \\ 
      Define 
      \begin{align}
        \label{eq:vn}
        v_m^{A,B} (\vt, \vl) = \inv{m} \sum_{j=1}^m f(G^j , N^{\mu,j})^2 \expp{-
        A\vt
        \cdot G^j + \frac{|A\vt|^2}{2}} \prod_{i=1}^p \expp{B\vl_i -\mu_i}
        \left(\frac{\mu_i}{B\vl_i}\right)^{N^{\mu_i,j}_i}.
      \end{align}
      Compute 
      \begin{align*}
        (\vt_m, \vl_m) = \arg\min_{(\vt, \vl) \in \R^d \times \Rp^p}
        v_m^{A,B}(\vt, \vl).
      \end{align*}

    \item[Second stage:] Generate a sequence $(\bar G^j)_{j = 1, \dots, n}$ of
      i.i.d random vector following the standard normal distribution in
      $\R^{d}$ and a sequence $(\bar N^j = (\bar N_1^j, \dots, \bar
      N_{p}^j))_{j = 1, \dots, n}$ of i.i.d Poisson random vectors with
      parameter $B\vl_m$. Conditionally on $\vl_m$, these two sequences are
      assumed to be independent of the sequences $(G^j)_{j= 1,
      \dots, m}$ and $(N^{\mu,j})_{j =1, \dots, m}$ \\ Define
      \begin{align}
        \label{eq:Mn}
        M^{A,B}_{n,m} = \frac{1}{n} \sum_{j=1}^n f(\bar G^j +A \vt_m, \bar N^j)
        \expp{- A\vt_m \cdot \bar G^j - \frac{|A\vt_m|^2}{2}}
        \prod_{i=1}^p \expp{(B\vl_m)_i -\mu_i}
        \left(\frac{\mu_i}{(B\vl_m)_i}\right)^{\bar N^j_i}.
      \end{align}
  \end{description}
\end{algo}

\subsection{Strong law of large numbers and central limit theorem}

The conditional independence between the two stages combined with
Lemma~\ref{lem:chgt} immediately shows that for any fixed $m,n$, the estimator
$M^{A,B}_{n,m}$ is unbiased, ie. $\E[M^{A,B}_{n,m}]= \ce$.  Conditionally on $(G_j,
N_j)_{j=1,\dots,m}$, the terms involved in the sum of Equation~\eqref{eq:Mn} are
i.i.d., hence the standard strong law of large numbers yields that $\lim_{n
\rightarrow +\infty} M^{A,B}_{n,m} = \E[f(G, N^\mu)]$ a.s. by applying
Lemma~\ref{lem:chgt}. Similarly, the central limit theorem applies and we can
state the following result.
\begin{prop}
  \label{prop:sln_tclfixed}
  For any fixed $m$, $M^{A,B}_{n,m}$ converges a.s. to $\E[f(G, N^\mu)]$ as $n$ goes to
  infinity and moreover $\sqrt{n} (M^{A,B}_{n,m} - \E[f(G, N^\mu)]) \xrightarrow[n
  \rightarrow +\infty]{law} \cn(0, v^{A,B}(\vt_m, \vl_m))$.
\end{prop}
This result is not fully satisfactory as from a practical point of view, we like
to know the limiting of the estimator $M^{A,B}_{n,m(n)}$ where $m(n)$ is a function of
$n$ tending to infinity with $n$.  To investigate the asymptotic behaviour when
$m$ and $n$ tend to infinity together, it is convenient to rewrite $M^{A,B}_{n, m(n)}$
using an auxiliary sequence of random variables.  We introduce a sequence $(\bar
U_i^j)_{1 \le i \le p, j \ge 1}$ of i.i.d. random variables following the
uniform distribution on $[0, 1]$ and independent of all the other random
variables used so far. If we define
\begin{align*}
  \tilde N_i^j(\l) = \sum_{k=0}^\infty k \ind{P(\l_i; k) \le U_i^j < P(\l_i;
  k+1)} \quad \mbox{for all } 1 \le i \le p, \; 1 \le j 
\end{align*}
where $P(\l, \cdot)$ is the cumulative distribution function of the Poisson
distribution with parameter $\l$, then $(\bar N^j)_{j=1,\dots,n}
\stackrel{Law}{=} (\tilde N^j(\l_{m(n)}))_{j=1,\dots,n}$. Since for all $k \in
\N$, the function $\l \in \R^* \longmapsto P(\l, k)$ is continuous and
decreasing, we get that
$\lim_{n \rightarrow \infty} \tilde N^j (\l_{m(n)}) = N^j (\l_\s)$ a.s. and
for all $\l \le \l'$, $\tilde N^j (\l') < \tilde N^j (\l)$ where the
ordering has to be understood component wise.

We define
\begin{align*}
  \tilde M_{n}(\t, \l) = \frac{1}{n} \sum_{j=1}^n f(\bar G^j +\t, \tilde
  N^j(\l)) \expp{- \t \cdot \bar G^j - \frac{|\t|^2}{2}}
  \prod_{i=1}^p \expp{\l_i -\mu_i}
  \left(\frac{\mu_i}{\l_i}\right)^{\tilde N^j_i(\l)}.
\end{align*}
It is obvious that $M^{A,B}_{n, m(n)} \stackrel{Law}{=} \tilde M_n (A\vt_{m(n)},
B\vl_{m(n)})$.

\begin{thm}
  \label{thm:sln} Let $m(n)$ be an increasing function of $n$ tending to
  infinity. Then, under Assumptions~\hypref{hyp0} and~\hypref{integrable2},
  $M^{A,B}_{n,m(n)}$ converges a.s. to $\E[f(G, N^\mu)]$ as $n$ goes to infinity. 
\end{thm}
It is actually sufficient to prove the result for $A$ and $B$ being identity
matrices. For the sake of clear notations, when $A = I_d$ and $B = I_p$, we write
$M_{n,m(n)}$ instead of  $M^{A,B}_{n,m(n)}$.

\begin{proof}
  We have already seen that $\E[M_{n, m}] = \ce$.
  Thanks the independence of the samples used in the two stages of
  the algorithm, conditionally on $((G^j, N^j), j \ge 1)$, $M_{n, m}$
  writes as a sum of i.i.d random variables. We introduce the $\sigma-$algebra
  $\cg = \sigma( (G^j, N^j), j \ge 1)$. 
  We define for all $m, j \ge 1$
  \begin{align*}
    X_{m, j}  = f(\bar G^j +\t_m, \bar N^j) \expp{- \t_m
    \cdot \bar G^j - \frac{|\t_m|^2}{2}}
    \prod_{i=1}^p \expp{(\l_m)_i -\mu_i}
    \left(\frac{\mu_i}{(\l_m)_i}\right)^{\bar N^j_i}.
  \end{align*}
  Note that conditionally on $\cg$, the sequence $(X_{m, j})_{j \ge 1}$ is
  i.i.d. for any fixed $m \ge 1$. 

  For a fixed $\varepsilon>0$, we recall the definition of $\cv_{\varepsilon}$
  \begin{equation*}
    \cv_{\varepsilon} \defeq \left\{ (\t, \l) \in \R^d \times
    \R^p \; : \; |(\t, \l) - (\t_\s, \l_\s)| \le \varepsilon \right\}.
  \end{equation*}
  In the following, we assume that $\varepsilon$ is small enough, so that
  $\cv_{\varepsilon}$ is included in $\R^d \times \Rp^p $.  For all $m, n \ge
  1$,
  \begin{align}
    \label{eq:nVar}
    & \E\left[ (M_{n,m} - \ce)^2 \ind{(\t_{m},\l_m) \in \cv_{\varepsilon}} \right]  = 
    \E\left[\E\left[ \left(\inv{n} \sum_{i=1}^n (X_{m,i}- \ce) \right)^2 \Big| \cg
    \right] \ind{(\t_{m},\l_m) \in \cv_{\varepsilon}}\right]  \nonumber \\
    & \le \inv{n}   \E\left[\E\left[(X_{m,i}- \ce)^2 \Big| \cg
    \right] \ind{(\t_{m},\l_m) \in \cv_{\varepsilon}}\right]  \nonumber \\
    & \le \inv{n}   \E\left[ v(\t_m, \l_m) \ind{(\t_{m},\l_m) \in \cv_{\varepsilon} }\right]  \nonumber \\
    & \le \inv{n} \left(\sup_{(\t, \l) \in \cv_{\varepsilon}}v(\t, \l)  - \ce^2\right) \le
    \frac{c}{n}.
  \end{align}
  We deduce from the Borell Cantelli Lemma that for any increasing function
  $\rho : \N \rightarrow \N$, $(M_{n^2,\rho(n)} - \ce)
  \ind{(\t_{\rho(n)},\l_{\rho(n)}) \in \cv_{\varepsilon}}$ tends to $0$ a.s. \\

  To prove that $(M_{n,m(n)} - \ce) \ind{(\t_{m(n)},\l_{m(n)}) \in \cv_{\varepsilon}}$
  converges to zero a.s., we mimic the proof of the classical strong law of
  large numbers.

  Let $n \in \N^*$, we define $k = \lfloor \sqrt{n} \rfloor$; then $k^2 \le n
  < (k+1)^2$. 
  \begin{align}
    \label{eq:mimicslln}
    M_{n,m(n)} - \ce = & \inv{n}
    \sum_{i=1}^{k^2} (X_{m(n), i}  - \ce) + \inv{n} \sum_{i=k^2+1}^{n} (X_{m(n),
    i} - \ce) \nonumber \\
    \abs{M_{n,m(n)} - \ce} & \le \inv{k^2}
    \abs{\sum_{i=1}^{k^2} (X_{m(n), i}  - \ce)} + \inv{n} \abs{\sum_{i=k^2+1}^{n} (X_{m(n),
    i} - \ce)}.
  \end{align}
  Using Equation~\eqref{eq:nVar}, 
  \begin{align*}
    \E\left[\left(\inv{k^2} \abs{\sum_{i=1}^{k^2} (X_{m(n), i}  - \ce)}
    \right)^2 \ind{(\t_{m},\l_m) \in \cv_{\varepsilon}}\right] \le \frac{c}{k^2}.
  \end{align*}
  Hence, we easily deduce from the Borrel Cantelli Lemma that $\inv{k^2}
  \abs{\sum_{i=1}^{k^2} (X_{m(n), i}  - \ce)} \ind{(\t_{m(n)},\l_{m(n)}) \in \cv_{\varepsilon}}$
  tends to $0$ a.s. when $k$ goes to infinity, ie. when $n$ goes to infinity.  A
  similar computation as in Equation~\eqref{eq:nVar} leads to 
  \begin{align*}
    \E \left[ \left( \inv{n} \abs{\sum_{i=k^2+1}^{n} (X_{m(n), i} - \ce)}
    \right)^2 \ind{(\t_{m(n)},\l_{m(n)}) \in \cv_{\varepsilon}}\right] \le
    \frac{n - k^2}{n^2}\left(\sup_{(\t, \l) \in K}v(\t, \l)  - \ce^2\right)  \le
    \frac{c}{n^{3/2}}.
  \end{align*}
  Hence, the Borel Cantelli Lemma yields that $\inv{n} \abs{\sum_{i=k^2+1}^{n}
  (X_{m(n), i} - \ce)} \ind{(\t_{m(n)},\l_{m(n)}) \in
  \cv_{\varepsilon}}\rightarrow 0$ a.s. when $n$ goes to infinity. 

  Eventually, we have proved that $(M_{n,m(n)} - \ce) \ind{(\t_{m(n)},\l_{m(n)})
  \in \cv_{\varepsilon}}$ converges to zero a.s. Since, $(\t_{m(n)}, \l_{m(n)})
  \rightarrow (\t_\s, \l_\s)$ a.s., we deduce that $M_{n,m(n)} \rightarrow \ce$
  a.s. when $n$ goes to infinity.
\end{proof}

\begin{thm}
  \label{thm:tcl}
  Let $m(n)$ be an integer valued function of $n$ increasing to infinity with $n$ and
  such that $m(n) \sim n^\beta$ for some $\beta > 0$. Assume that
  \begin{hypo}
    \label{hyp:ulln_fIS}
    \begin{subhypo}
      \item for all $k \in \N^p$, the function $g \in \R^{d} \longmapsto f(g, k)$ is
        continuous;
      \item there exists a compact neighbourhood $\cv$ of $(\vt_\s, \vl_\s)$ included
        in $\R^{d'} \times {\R_+^*}^{p'}$ and $\eta >0$ such that $\E\left[ \sup_{(\vt,
        \vl) \in \cv} |f(\bar G+A\vt, \tilde N^1(B\vl))|^{2(1 + \eta)} \right] <
        \infty$.
      \end{subhypo}
  \end{hypo}
  Then, under Assumptions~\hypref{hyp0} and~\hypref{integrable2}, 
  $$
  \sqrt{n}
  (\tilde M_{n}(A\vt_{m(n)}, B\vl_{m(n)}) - \E[f(G, N^\mu)]) \xrightarrow[n
  \rightarrow +\infty]{law} \cn(0, v^{A,B}(\vt_\s, \vl_\s)).
  $$
\end{thm}

\begin{proof}
  It is actually sufficient to prove the result for $A$ and $B$ being identity
  matrices.
  \begin{equation*}
    \sqrt{n} (\tilde M_n (\theta_{m(n)}, \l_{m(n)}) - \ce) = \sqrt{n} (\tilde M_{n}(\t_\s, \l_\s) - \ce) 
    + \sqrt{n} (\tilde M_{n} (\t_{m(n)}, \l_{m(n)}) - \tilde M_{n}(\t_\s, \l_\s)) 
  \end{equation*}
  From the standard central limit theorem, $\sqrt{n} (\tilde M_{n}(\t_\s,
  \l_\s) - \ce) \xrightarrow[n \rightarrow +\infty]{law} \cn(0, v(\t_\s,
  \l_\s))$. Therefore, it is sufficient to prove that $\sqrt{n} (\tilde
  M_{n}(\t_{m(n)}, \l_{m(n)}) - \tilde M_{n}(\t_\s, \l_\s)) \xrightarrow[n
  \rightarrow +\infty]{Pr} 0$. Let $\varepsilon>0$ and $0 < \alpha < \beta/2$.
  \begin{multline*}
    \P\left( \sqrt{n} \abs{\tilde
    M_{n}(\t_{m(n)}, \l_{m(n)}) - \tilde M_{n}(\t_\s, \l_\s)} > \varepsilon
    \right) \le  \P( n^\alpha\abs{(\t_{m(n)}, \l_{m(n)}) - (\t_\s,
    \l_\s)} > 1) \\
    + \frac{n}{\varepsilon^2} \E\left[ \abs{\tilde M_{n}(\t_{m(n)},
    \l_{m(n)}) - \tilde M_{n}(\t_\s, \l_\s)}^2 \ind{\abs{(\t_{m(n)}, \l_{m(n)})
    - (\t_\s, \l_\s)} \le n^{-\alpha} } \right].
  \end{multline*}
  Note that $n^\alpha \sim m(n)^{\alpha/\beta}$ with $\alpha / \beta < 1/2$,
  hence we deduce from Proposition~\ref{prop:convsaa}, that $\P(
  n^\alpha\abs{(\t_{m(n)}, \l_{m(n)}) - (\t_\s, \l_\s)} > 1) \longrightarrow
  0$. We define
  \begin{align*}
    Q(\t, \l) = \expp{- \t \cdot \bar G^1 - \frac{|\t|^2}{2}}
    \prod_{i=1}^p \expp{\l_i -\mu_i}
    \left(\frac{\mu_i}{\l_i}\right)^{\tilde N^1_i(\l)}.
  \end{align*}
  Conditionally on $(\t_{m(n)}, \l_{m(n)})$, $\tilde M_{n}(\t_{m(n)},
  \l_{m(n)})$ writes as a sum of i.i.d random variables.
  \begin{align}
    \label{eq:cvL1Mtilde}
    &n \E\left[ \abs{\tilde M_{n}(\t_{m(n)}, \l_{m(n)}) -
    \tilde M_{n}(\t_\s, \l_\s)}^2 \ind{\abs{(\t_{m(n)}, \l_{m(n)}) - (\t_\s,
    \l_\s)} \le n^{-\alpha} } \right] =  \nonumber\\
    & \E\Bigg[ \left|
    f(\bar G^1 +\t_\s, \tilde N^1(\l_\s)) Q(\t_\s, \l_\s) -
    f(\bar G^1 +\t_{m(n)}, \tilde
    N^1(\l_{m(n)})) Q(\t_{m(n)}, \l_{_m(n)}) \right|^2 \nonumber\\
    &\qquad \ind{\abs{(\t_{m(n)}, \l_{m(n)}) - (\t_\s,
    \l_\s)} \le n^{-\alpha} } \Bigg].
  \end{align}
  Thanks to the convergence of $\tilde N^1(\l_{m(n)})$, $Q(\t_{m(n)}, \l_{m(n)})$
  converges a.s. to $Q(\t_\s, \l_\s)$ when $n$ goes to infinity. Since for
  $n$ large enough, $N^1(\l_{m(n)}) = N^1(\l_\s)$, the continuity of $f$ with
  respect to its first argument enables to prove that $f(\bar G^1 +\t_{m(n)}, \tilde
  N^1(\l_{m(n)}))$ converges a.s. to $f(\bar G^1 +\t_\s, \tilde N^1(\l_\s))$.
  Hence, the absolute value inside the expectation tends to zero a.s. We need to
  bound the term inside the expectation by an integrable random variable to
  apply the bounded convergence theorem which yields the result.
  
  \begin{align*}
    &\left|
    f(\bar G^1 +\t_\s, \tilde N^1(\l_\s)) Q(\t_\s, \l_\s) -
    f(\bar G^1 +\t_{m(n)}, \tilde
    N^1(\l_{m(n)})) Q(\t_{m(n)}, \l_{_m(n)}) \right|^2 \\
    &\qquad \ind{\abs{(\t_{m(n)}, \l_{m(n)}) - (\t_\s,
    \l_\s)} \le n^{-\alpha} }  \le 
    2 \sup_{|(\theta, \l) - (\t_\s, \l_\s)| \le n^{-\alpha}}
    \left| f(\bar G^1 +\t, \tilde N^1(\l)) \right|^2 Q^2(\t, \l).
  \end{align*}
  For $n$ large enough,  $\left\{ |(\theta, \l) - (\t_\s, \l_\s)| \le
  n^{-\alpha} \right\} \subset \cv$. Moreover, there exist $\underline{m}>0$ and
  $M>0$ such that $\cv \subset \left\{ |\t| \le M, |\l| \le M \mbox{ and }
  d_0(\l) \ge \underbar{m} \right\}$. Hence,
  \begin{align*}
    & \sup_{(\theta, \l) - (\t_\s, \l_\s)| \le n^{-\alpha}}
    \left| f(\bar G^1 +\t, \tilde N^1(\l)) \right| Q(\t, \l) \\
    & \le \sup_{(\theta, \l) \in \cv} \left| f(\bar G^1 +\t, \tilde N^1(\l))
    \right| \expp{p M} \prod_{i=1}^d (\expp{-M G^1_i} + \expp{M G^1_i})
    \prod_{i=1}^p \left(\left(\frac{\mu_i}{\underline{m}}\right)^{\tilde
    N^1_i(\underline{m})} + \left(\frac{\mu_i}{\underline{m}}\right)^{\tilde
    N^1_i(M)} \right) \\
    & \le  \sum_{\substack{\sigma \in \{-M,M\}^d \\ \nu \in \{\underline{m},
    M\}^p}} \sup_{(\theta, \l) \in \cv} \left| f(\bar G^1 +\t, \tilde N^1(\l))
    \right| \expp{p M} \expp{\sigma \cdot G^1} \prod_{i=1}^p
    \left(\frac{\mu_i}{\underline{m}}\right)^{\tilde N^1_i(\nu)}.
  \end{align*}
  Then, using Hölder's inequality we get
  \begin{align*}
    & \E \left[ \sum_{\substack{\sigma \in \{-M,M\}^d \\ \nu \in
    \{\underline{m}, M\}^p}} \sup_{(\theta, \l) \in \cv} \left| f(\bar G^1 +\t,
    \tilde N^1(\l)) \right|^2 \left( \expp{p M} \expp{\sigma \cdot G^1}
    \prod_{i=1}^p \left(\frac{\mu_i}{\underline{m}}\right)^{\tilde
    N^1_i(\nu)} \right)^2 \right] \\
   & \le \sum_{\substack{\sigma \in \{-M,M\}^d \\ \nu \in \{\underline{m},
   M\}^p}} \E \left[ \sup_{(\theta, \l) \in \cv} \left| f(\bar G^1 +\t, \tilde
   N^1(\l)) \right|^{2(1+\eta)} \right]^{\frac{1}{1+\eta}}
   \E\left[\left(\expp{p M} \expp{\sigma \cdot G^1} \prod_{i=1}^p
   \left(\frac{\mu_i}{\underline{m}}\right)^{\tilde N^1_i(\nu)}\right)^{2 +
   \frac{2}{\eta}} \right]^{\frac{\eta}{1+\eta}}. 
  \end{align*}
  Since we have assumed that $\E \left[ \sup_{(\theta, \l) \in \cv} \left|
  f(\bar G^1 +\t, \tilde N^1(\l)) \right|^{2(1+\eta)} \right] < \infty$, the
  random variables $\left| f(\bar G^1 +\t_\s, \tilde N^1(\l_\s)) Q(\t_\s,
  \l_\s) - f(\bar G^1 +\t_{m(n)}, \tilde N^1(\l_{m(n)})) Q(\t_{m(n)},
  \l_{_m(n)}) \right|^2 \ind{\abs{(\t_{m(n)}, \l_{m(n)}) - (\t_\s, \l_\s)}
  \le n^{-\alpha} }$ are uniformly bounded w.r.t $n$ by an integrable random
  variable. Hence, the left hand side of Equation~\eqref{eq:cvL1Mtilde} tends to
  zero which achieves to prove that $\sqrt{n} (\tilde M_{n}(\t_{m(n)},
  \l_{m(n)}) - \tilde M_{n}(\t_\s, \l_\s)) \xrightarrow[n \rightarrow
  +\infty]{Pr} 0$.
\end{proof}

\subsection{Practical implementation}

The difficult part of Algorithm~\ref{algo:is} is the numerical computation of
the minimizing pair $(\t_m, \l_m)$.  The efficiency of the optimization
algorithm depends very much on the magnitude of the smallest eigenvalue of
$\nabla^2 v$. From the end of the proof of Proposition~\ref{prop:convex}, we can
deduce that the smallest eigenvalue of $\nabla^2 v$ is larger than
\begin{align*}
  \E\left[F(G,\t,N^\mu,\l) \ind{N^\mu=(n_1,\dots,n_p)}\right] \min\left(1,
  \frac{n_1}{\l_1^2}, \dots, \frac{n_p}{\l_p^2} \right).
\end{align*}
This lower bound depends on the function $f$ whereas we would rather find a
uniform lower bound. Hence, we advice to rewrite $\nabla v$ as
\begin{align*}
  \nabla v(\t, \l) = & \E\left[ 
  \begin{pmatrix}
    \t  \\
    1_p 
  \end{pmatrix}
  f(G, N^\mu)^2 \expp{- \t \cdot G + \frac{|\t|^2}{2}}
  \prod_{i=1}^p \expp{\l_i -\mu_i} \left(\frac{\mu_i}{\l_i}\right)^{N^{\mu_i}_i}
  \right] \\
  & - \E\left[ 
  \begin{pmatrix}
    G \\
    \frac{N^\mu}{\l}
  \end{pmatrix}
  f(G, N^\mu)^2 \expp{- \t \cdot G + \frac{|\t|^2}{2}}
  \prod_{i=1}^p \expp{\l_i -\mu_i} \left(\frac{\mu_i}{\l_i}\right)^{N^{\mu_i}_i}
  \right]
\end{align*}
where $\frac{N\mu}{\l} = \left( \frac{N_1^{\mu_1}}{\l_1}, \dots,
\frac{N_p^{\mu_p}}{\l_p} \right)$. Hence, $(\t^\s, \l^\s)$ can be seen
as the root of 
\begin{align*}
  \nabla u(\t, \l) = & 
  \begin{pmatrix}
    \t  \\ 1_p 
  \end{pmatrix}
  - \frac{\E\left[ 
  \begin{pmatrix}
    G \\ \frac{N^\mu}{\l}
  \end{pmatrix}
  f(G, N^\mu)^2 \expp{- \t \cdot G}
  \prod_{i=1}^p \left(\frac{\mu_i}{\l_i}\right)^{N^{\mu_i}_i}
  \right]}{ \E\left[ f(G, N^\mu)^2 \expp{- \t \cdot G}
  \prod_{i=1}^p \left(\frac{\mu_i}{\l_i}\right)^{N^{\mu_i}_i}
  \right]}
\end{align*}
with $u(\t, \l) = \frac{|\t|^2}{2} + \sum_{i=1}^p \l_i + \log \E \left[
f(G, N^\mu)^2 \expp{- \t \cdot G}
\prod_{i=1}^p \left(\frac{\mu_i}{\l_i}\right)^{N^{\mu_i}_i}
\right]$. The Hessian matrix of $u$ is given by
\begin{align*}
  \nabla^2 u(\t, \l) = & 
  \begin{pmatrix}
    I_d & 0 \\ 0 & \frac{\E\left[ D f(G, N^\mu)^2 \expp{- \t \cdot G + }
  \prod_{i=1}^p \left(\frac{\mu_i}{\l_i}\right)^{N^{\mu_i}_i}
  \right]}{\E\left[ f(G, N^\mu)^2 \expp{- \t \cdot G}
  \prod_{i=1}^p \left(\frac{\mu_i}{\l_i}\right)^{N^{\mu_i}_i}
  \right]}
  \end{pmatrix} \\
  & + \frac{\E\left[ 
  \begin{pmatrix}
    G \\ \frac{N^\mu}{\l}
  \end{pmatrix}
  \begin{pmatrix}
    G \\ \frac{N^\mu}{\l}
  \end{pmatrix}^*
  f(G, N^\mu)^2 \expp{- \t \cdot G}
  \prod_{i=1}^p \left(\frac{\mu_i}{\l_i}\right)^{N^{\mu_i}_i}
  \right]}{ \E\left[ f(G, N^\mu)^2 \expp{- \t \cdot G}
  \prod_{i=1}^p \left(\frac{\mu_i}{\l_i}\right)^{N^{\mu_i}_i}
  \right]} \\
  & -  \frac{\E\left[ 
  \begin{pmatrix}
    G \\ \frac{N^\mu}{\l}
  \end{pmatrix}
  f(G, N^\mu)^2 \expp{- \t \cdot G}
  \prod_{i=1}^p \left(\frac{\mu_i}{\l_i}\right)^{N^{\mu_i}_i}
  \right]\E\left[ 
  \begin{pmatrix}
    G \\ \frac{N^\mu}{\l}
  \end{pmatrix}
  f(G, N^\mu)^2 \expp{- \t \cdot G}
  \prod_{i=1}^p \left(\frac{\mu_i}{\l_i}\right)^{N^{\mu_i}_i}
  \right]^* }{ \E\left[ f(G, N^\mu)^2 \expp{- \t \cdot G}
  \prod_{i=1}^p \left(\frac{\mu_i}{\l_i}\right)^{N^{\mu_i}_i}
  \right]^2}
\end{align*}
where we recall that the diagonal matrix $D$ is defined by $D = \diag_p \left(
\frac{N_1^{\mu_1}}{\l_1^2}, \dots, \frac{N_p^{\mu_p}}{\l_p^2} \right)$.
The Cauchy Schwartz inequality yields that the last two terms in the expression
of $\nabla^2 u$ form a positive semi definite matrix. The first part of the
Hessian is a positive definite matrix with smallest eigenvalue larger than
\begin{align*}
 & \min\left(  1, \frac{1}{\l_j^2} 
  \frac{\E\left[ N_i^{\mu_i} f(G, N^\mu)^2 \expp{- \t \cdot G}
  \prod_{i=1}^p \left(\frac{\mu_i}{\l_i}\right)^{N^{\mu_i}_i}
  \right]}{\E\left[ f(G, N^\mu)^2 \expp{- \t \cdot G}
  \prod_{i=1}^p \left(\frac{\mu_i}{\l_i}\right)^{N^{\mu_i}_i}
  \right]}
 \right) \\
& \qquad = \min\left(  1, \frac{\mu_j}{\l_j^3} 
  \frac{\E\left[ f(G, N^\mu + e_j)^2 \expp{- \t \cdot G}
  \prod_{i=1}^p \left(\frac{\mu_i}{\l_i}\right)^{N^{\mu_i}_i}
  \right]}{\E\left[ f(G, N^\mu)^2 \expp{- \t \cdot G}
  \prod_{i=1}^p \left(\frac{\mu_i}{\l_i}\right)^{N^{\mu_i}_i}
  \right]}
 \right)
\end{align*}
where the equality comes from Stein's formula for Poisson random variables and
$e_j$ denotes the $j-th$ element of the canonical basis. When the function $f$
is increasing with respect to each component of its second argument, then we
come up with the following lower bound independent of the function $f$
\begin{align*}
  \min\left(  1, \frac{\mu_j}{\l_j^3} \right).
\end{align*}
Our numerical experiments advocate the use of $u$ instead of $v$ to speed up the
computation of $(\t^\s, \l^\s)$.

Using this new expression, we implement Algorithm~\ref{newton-algo} to
construct an approximation $x_n^k$ of $(\t_n, \l_n)$. Since $ u_n$ is
strongly convex, for any fixed $n$, $x_n^k$ converges to $(\t_n, \l_n)$ when
$k$ goes to infinity.  The direction of descent $d_n^k$ at step $k$ should be
computed as the solution of a linear system. There is no point in computing
the inverse of $\nabla^2 u_n(x_n^k)$, which would be computationally
much more expensive.

{\bf Remarks on the implementation :} From a practical point of view,
$\varepsilon$ should be chosen reasonably small  $\varepsilon \approx
10^{-6}$. This algorithm converges very quickly and, in most cases, less than
$5$ iterations are enough to get a very accurate estimate of $(\t_n, \l_n)$,
actually within the $\varepsilon-$error. Since the points at which the 
function $f$ is evaluated remain constant through the iterations of Newton's
algorithm, the values $f^2(G^j, N^j)$ for $j=1, \dots, n$ should be precomputed
before starting the optimization algorithm which considerably speeds up the
whole process. The Hessian matrix of our problem is easily tractable so
there is no point in using Quasi-Newton's methods.
\begin{algorithm}[!h]
  \caption{Projected Newton's algorithm}
  \label{newton-algo}
  \begin{algorithmic}
    \STATE Choose an initial value $x_n^0 \in \R^{d + p}$.
    \STATE $k=1$
    \WHILE{$\abs{\nabla u_n(x_n^k)} > \varepsilon$}
      \STATE 1. Compute $d_n^k$ such that $(\nabla^2 u_n(x_n^k)) d_n^k =
    - \nabla u_n(x_n^k)$
      \STATE 2. $x_n^{k+1/2} = x_n^k + d_n^k$
      \FOR{$i=1:d+p$}
        \IF{$x_n^{k+1/2}(i) > 0$} 
          \STATE  $x_n^{k+1}(i) = x_n^{k+1/2}(i)$
        \ELSE  
        \STATE $x_n^{k+1}(i) = \frac{x_n^{k}(i)}{2}$
        \ENDIF
      \ENDFOR
      \STATE 3. $k=k+1$
    \ENDWHILE
  \end{algorithmic}
\end{algorithm}

\section{Application to jump processes in finance}
\label{sec:apps}

We will apply our methodology to two different classes of jump processes:
jump diffusion processes and stochastic volatility processes with jumps, in this
latter case the volatility itself may jump also.

We consider a filtered probability space $(\Omega, \ca, (\cf_t)_{0 \le t \le T},
\P)$ with a finite time horizon $T>0$ and $I$ financial assets. We define on
this space a Brownian motion $W$ with values in $\R^I$ and $I+1$ independent
Poisson processes $(N^1, \dots, N^{I+1})$ with constant intensities $\mu^1,
\dots, \mu^{I+1}$. We also consider $(I+1)$ independent sequences
$(Y^i_j)_{j \ge 1}$ for $i=1 \dots I+1$ of i.i.d. real valued random variables
with common law denoted $Y$ in the following.  The Poisson processes, the
Brownian motions and the sequences $(Y^i_j)_{j}$ are supposed to be independent
of each other.  Actually, we are interested in considering the compound Poisson
process associated to the Poisson process $N^i$ and to the jump sequences $Y^i$
for $i=1,\dots,I+1$.

\subsection{Jump diffusion processes}

In this class of models, we assume that the log-prices evolve according to the
following equation
\begin{align}
  \label{eq:log-merton}
  X^i_t = \left(\beta^i - \frac{{(\sigma^i)}^2}{2}\right) t + \sigma^i L^i W_t +
  \sum_{j=1}^{N^i_t} Y^i_j + \sum_{j=1}^{N^{I+1}_t} Y^{I+1}_j
\end{align}
where $\beta = (\beta^i, \dots, \beta^I)^*$ is the drift vector and $\sigma =
(\sigma^i, \dots, \sigma^I)^*$ the volatility vector. The row vectors $L_i$ are
such that the matrix $L = (L^1; \dots; L^I)$ verifies that $\Gamma =
L L^*$ is a symmetric definite positive matrix with unit diagonal elements.  The matrix
$\Gamma$ embeds the covariance structure of the continuous part of the model. We
have also chosen to take into account in the model the possibility to have
simultaneous jumps which explains the extra jump term $\sum_{j=1}^{N^{I+1}_t}
Y^{I+1}_j$ common to all underlying assets. This common jump term embeds the
systemic risk of the market.

From Equation~\ref{eq:log-merton}, we deduce that the prices at time $t$ 
$S^i_t = \expp{X^i_t}$ are defined by
\begin{align*}
  S^i_t = S^i_0 \exp\left\{ \left(\beta^i - \frac{{(\sigma^i)}^2}{2}\right) t +
  \sigma^i L^i W_t \right\} 
  \prod_{j=1}^{N^i_t} \expp{Y^i_j}  \prod_{j=1}^{N^{I+1}_t} \expp{Y^{I+1}_j}
\end{align*}
which corresponds for each asset to a one dimensional Merton model with
intensity $\mu^i + \mu^{I+1}$ when the $Y^i_j$ are normally distributed.

As, we assumed that $\P$ was the martingale measure associated to the risk free rate
$r >0$ supposed to be deterministic, the processes $(\expp{-r t} S_t)_t$
must be martingales under $\P$. This martingale condition imposes that for every
$i=1,\dots,I$, 
\begin{align*}
  \beta^i = r - (\mu^i \E[Y^i] + \mu^{I+1} \E[Y^{I+1}]).
\end{align*}
In the following, $\beta_i$ will always stand for this quantity.

\begin{rem}
  In the one dimensional case, ie. when $I=1$, we only consider a single
  compound Poisson process as the systemic risk jump term becomes irrelevant.
  Hence, the log-price in dimension one will follow
  \begin{align*}
    X_t = \left(\beta - \frac{{\sigma}^2}{2}\right) t + \sigma W_t +
    \sum_{j=1}^{N_t} Y_j.
  \end{align*}
  For the sake of clearness, we will not treat the one dimensional case separately in the
  following, even though the practical one dimensional implementation relies on a single
  Poisson process. So, we will always consider that the Poisson process has values in
  $\R^{I+1}$.
\end{rem}

In the numerical examples, we will need to discretize the multi dimensional
price process on a time grid $0 = t_0 < t_1 < \dots < t_J = T$. We will assume
that this time grid is regular and given by $t_j = \frac{j T}{J}$,
$j=0,\dots,J$. Just to fix our notations, we consider that the Brownian (resp.
Poisson) increments are stored as a column vector with size $I \times J$ (resp.
$(I+1) \times J$).
$$\begin{pmatrix}
  W_{t_1} \\ W_{t_2}\\ \vdots \\ W_{t_{J-1}} \\ W_{t_J}
\end{pmatrix}=
\begin{pmatrix}
  \sqrt{t_1} Id & 0 & 0 &\hdots &0\\
  \sqrt{t_1} Id &\sqrt{t_2-t_1} Id & 0 &\hdots &0\\
  \vdots&\ddots&\ddots&\ddots&\vdots\\
  \vdots&\ddots&\ddots& \sqrt{t_{J-1}-t_{J-2}} Id&   0 \\
  \sqrt{t_1}Id & \sqrt{t_2-t_1} Id  &\hdots & \sqrt{t_{J-1}-t_{J-2}} Id
  &\sqrt{t_J-t_{J-1}} Id
\end{pmatrix}G,$$
where $G$ is a normal random vector in $\R^{I \times J}$ and  $Id$ is the
identity matrix in dimension $I \times I$. The Poisson process is discretized in
a similar way.

\paragraph{The Merton jump diffusion model.}

The Merton model corresponds to the particular choice of a normal distribution
for the variables $(Y^i)$, $Y^i \sim \cn(\alpha, \delta)$ where $\alpha \in \R$
and $\delta>0$. In this framework, the jump sizes in the price follow a log
normal distribution.

\paragraph{The Kou model.}

In the Kou model~\cite{Kkou02}, the variables $Y^i$ follow an asymmetric exponential
distribution with density
\begin{align*}
  p^i \mu^i_+ \expp{-\mu^i_+ x} \ind{x >0} + (1-p)^i \mu^i_-
  \expp{\mu^i_- x} \ind{x <0} 
\end{align*}
where $p^i \in [0,1]$ is the probability of a positive jump for the $i-th$
component and the variables $\mu^i_+ >0, \mu^i_->0$ govern the decay of
each exponential part.

\subsection{Stochastic volatility models with jumps}


In this section, we consider the stochastic volatility type model developed
by~\cite{bns01-2,bns01-1} in which the volatility process is a non Gaussian
Ornstein Uhlenbeck driven by a compound Poisson process.

We consider that the log-prices satisfy for $i=1,\dots,I$
\begin{align*}
  dX^{i}_t = (a^i - \sigma^i/2) dt + \sqrt{\sigma^i_{t^-}} dW^i_t + \psi^i dZ^i_{\kappa^i t} +
  \psi^{I+1} dZ^{I+1}_{\kappa^{I+1} t}
\end{align*}
where $a \in \R^I$, $\psi \in \R^{I+1}$ has non-positive components which
account for the positive leverage effect, $Z$ is $(I+1)$-dimensional Lévy
process defined by $Z^i_t = \sum_{k=1}^{N_t^i} Y^i_k$ for $i=1,\dots,I+1$
and the squared volatility process $(\sigma_t)_t$ is Lévy driven Ornstein Uhlenbeck
\begin{align*}
  d\sigma^i_t = -(\kappa^i  + \kappa^{I+1}) \sigma^i_t dt + dZ^{i}_{\kappa^i t} +
  dZ^{I+1}_{\kappa^{I+1} t}.
\end{align*}
For the squared volatility process to remain positive, we assume that the
components of $Z$ only jumps upward, which means that the random variables
$Y^i_j$ are non-negative. 

More specifically, the jump sequence $Y^i$ is i.i.d following the exponential
distribution with parameter $\beta^i>0$ for $i=1,\dots,I+1$. The drift vector
$a$ is chosen such that the discounted prices are martingales under $\P$. Hence,
a straight computation shows that we need to set 
\begin{align*}
  a^{i} = r - \psi^i \frac{\kappa^i \mu^i}{\beta^i - \psi^i} - \psi^{I+1}
  \frac{\kappa^{I+1} \mu^{I+1}}{\beta^{I+1} - \psi^{I+1}}, \quad
  \mbox{for $i=1,\dots,I$} 
\end{align*}
to ensure the martingale property of $(\expp{-rt} \exp{X_t})_t$.

As in the section on jump diffusion models, the extra Poisson process giving
raise to the term $dZ^{I+1}$ in the dynamics of $X$ and $\sigma$ accounts for
modelling a systemic risk. When $Z^{I+1}$ jumps, all the volatilities and
possibly all the assets (when there is a leverage effect) jump
together. This parametrization of multi-dimensional stochastic volatility models
with jumps corresponds to Section 5.3 of \cite{BNS12}. Adding this extra jump
process only makes sense in a multi-dimensional framework, hence we write the
one-dimensional model using the previous equations but without the terms
involving the index $I+1$.

In the following, we compare the efficiencies of several different approaches
based on the theoretical part of the paper in the context of option pricing with
jumps. The problem always boils down to computing the expectation of a function
of a jump diffusion process. 

\subsection{Several importance sampling approaches}

To design an importance sampling Monte Carlo
method, we can either play with the Brownian part --- referred to hereafter as
\textit{Gaussian importance sampling} with an optimal variance denoted
\textit{VarG}, or with the Poisson part --- referred to as \textit{Poisson
importance sampling} with an optimal variance denoted \textit{VarP}, or with
both at the same time. This last approach is named \textit{Gaussian+Poisson
importance sampling} and yields to an optimal variance denoted \textit{VarGP}.
The Gaussian importance sampling approach actually corresponds the methodology
developed in~\cite{jour-lelo-09} but with independent sets of samples for the
optimization part and the true Monte Carlo computation. \\

For each of the three methods, we consider two approaches.

\paragraph{Full importance sampling.} The first approach consists in allowing to
optimize the parameters per time steps, this means that $d = d' =  I \times J$
and $p = p' = (I+1) \times J$. In this setting, the matrices $A$ and $B$ are
identity matrices.  This approach is the more general in one framework, but the
dimension of the optimization problem linked to the variance minimization with
the square of the number of time steps, which yields some interest in trying to
find a sub vector space with smaller dimension in which optimizing the
parameters and which achieves a variance close the global minimum.

\paragraph{Reduced importance sampling.}
The idea of reducing the dimension of the problem is to search for the
parameter $(\t, \l)$ in the subspace $\{ (A \vt, B \vl) \; : \;
\vt \in \R^{d'}, \; \vl \in \Rp^{p'} \}$ where  $A \in \R^{d \times d'}$ is a
matrix with rank $d' \le d$ and $B \in \Rp^{p \times p'}$ a matrix with rank $p'
\le p$. 

We choose to restrict ourselves to adding a constant drift to the Brownian motion
and keeping the Poisson intensity time independent. This corresponds to $d' = I$
and $p' = I+1$ 
$$
A_{(j-1)I+i, i} = \sqrt{t_j-t_{j-1}}, \quad B_{(j-1)(I+1)+k, k} = t_j-t_{j-1}
$$ 
for $j=1,\dots, J$, $i=1, \dots, I$ and $k=1, \dots,I+1$ , all the other
coefficients of $A$ and $B$ being zero.

\subsection{Numerical experiments}

We compare the different importance sampling approaches on four different
financial derivatives: the first two examples are path-dependent single asset
options while the last two examples are basket option with or without barrier
monitoring. To compare the different strategies, we have decided to fix the
number of samples for the Monte Carlo part, which implies that their accuracies
only depend on their variances, which we will compare in different examples. To
determine which method is best, it is convenient to compute their efficiencies
defined as the ratio of the variance divided by the CPU time. \\

In all the following examples, we use the same number of samples for the
approximation of the optimal importance sampling parameters and for the Monte
Carlo computation, ie.  $m(n) = n$.

\paragraph{Asian option.} We consider a discretely monitored Asian option with
payoff 
\begin{align*}
  \left( \frac{1}{J} \sum_{i=1}^J S_{t_i} - K \right)_+.
\end{align*}

Our tests on one dimensional Asian options (see Tables~\ref{tab:asian_1d_merton}
and~\ref{tab:asian_1d_bns}) show that the \textit{Poisson} and
\textit{Gaussian+Poisson} importance sampling methods perform generally better
than the pure \textit{Gaussian} importance sampling approach but they also
require a longer computational time. When taking into account this extra
computational times along with the variance reduction we notice that the
\textit{Poisson} and \textit{Gaussian+Poisson} importance sampling methods yield
the same efficiency for the Merton model (see Table~\ref{tab:asian_1d_merton}).
For the BNS model (see Table~\ref{tab:asian_1d_bns}), the mixed
\textit{Gaussian+Poisson} importance sampling approach achieves a better
variance reduction than the two other methods for a comparable computational
time. By closely looking at the CPU times of the different strategies, it
clearly appears that the reduced approach shows the better efficiency and should
be used in practice.

\begin{table}[h]
  \centering
  \begin{tabular}{lllllll}
    \hline
    & Strike & Price & Var & VarG & VarP & VarGP \\
    \hline
    Full & 90  & 17.88 & 2639 & 2395 & 636 & 529  \\
    Reduced  & & 17.88 & 2639 & 2640 &  839 & 752 \\
    Full & 100 & 14.37 & 2750 & 2624 & 720 & 622 \\
    Reduced &  & 14.37 & 2750 & 2624 & 552 & 470 \\
    Full & 110 & 12.11 & 2327 & 2301 & 470 & 420 \\
    Reduced &  & 12.11 & 2327 & 2301 & 676 & 585 \\
    \hline
  \end{tabular}
  \caption{Discrete Asian option in dimension 1 in the Merton model with
  $S_0=100$, $r=0.05$, $\sigma=0.25$, $\mu=1$, $\alpha=0.5$, $\delta=0.2$,
  $T=1$, $J=12$ and $n=50000$. The CPU time for the crude Monte Carlo approach
  is $0.08$.  The CPU times for the full importance sampling approach are $(0.21,
  0.28, 0.39)$ and for the reduced approach $(0.20, 0.21, 0.26)$. }
  \label{tab:asian_1d_merton}
\end{table}

\begin{table}[h]
  \centering
  \begin{tabular}{lllllll}
    \hline
    & Strike & Price & Var & VarG & VarP & VarGP \\
    \hline
    Full & 90  & 11.85 & 63 & 22.7 & 50 & 13.3  \\
    Reduced  & & 11.85 & 63 & 28.7 &  52.7 & 22.1 \\
    Full & 100 & 3.96 & 47 & 19 & 29.7 & 9.4 \\
    Reduced &  & 3.96  & 47 & 22 & 33 & 14.7 \\
    Full & 110 & 0.92 & 19 & 7.8 & 9 & 3.5 \\
    Reduced &  & 0.92 & 19 & 10 & 11.1 & 5.56 \\
    \hline
  \end{tabular}
  \caption{Discrete Asian option in dimension 1 in the BNS model with
  $S_0=100$, $r=0.05$, $\vl_0=0.01$, $\mu=1$, $\kappa=0.5474$, $\beta=18.6$,
  $T=1$, $J=12$ and $n=50000$. The CPU time for the crude Monte Carlo approach
  is $0.13$.  The CPU times for the full importance sampling approach are $(0.36,
  0.52, 0.93)$ and for the reduced approach $(0.29, 0.29, 0.33)$. }
  \label{tab:asian_1d_bns}
\end{table}

\paragraph{Barrier option.} We consider a discrete monitoring barrier option
with payoff
\begin{align*}
  (S_T - K)_+ \times \ind{\forall 1 \le j \le J, \; S_{t_j} < U}
\end{align*}
where $U$ is the upper barrier.

\begin{table}[h]
  \centering
  \begin{tabular}{lllllll}
    \hline
    & Strike & Price & Var & VarG & VarP & VarGP \\
    \hline
    Full & 90  & 17.88 & 2639 & 2395 & 636 & 529  \\
    Reduced  & & 17.88 & 2639 & 2640 &  839 & 752 \\
    Full & 100 & 14.37 & 2750 & 2624 & 720 & 622 \\
    Reduced &  & 14.37 & 2750 & 2624 & 552 & 470 \\
    Full & 110 & 12.11 & 2327 & 2301 & 470 & 420 \\
    Reduced &  & 12.11 & 2327 & 2301 & 676 & 585 \\
    \hline
  \end{tabular}
  \caption{Discrete barrier option in dimension $1$ in the Merton model with
  $S_0=100$, $r=0.05$, $\sigma=0.2$, $\mu=0.1$, $\alpha=0$, $\delta=0.1$,
  $T=1$, $J=12$, $U=140$ and $n=50000$. The CPU time for the crude Monte Carlo
  approach is $0.08$.  The CPU times for the full importance sampling approach are
  $(0.22, 0.26, 0.37)$ and for the reduced approach $(0.20, 0.20, 0.23)$. }
  \label{tab:barrier_1d}
\end{table}

\paragraph{Basket option.} We consider a basket option on $10$ assets with
payoff 
\begin{align*}
  \left( \sum_{i=1}^I \omega^i S^i_T - K \right)_+
\end{align*}
where the vector $\omega \in \R^{I}$ describes the weight of each asset in the
basket. 

\begin{table}[h]
  \centering
  \begin{tabular}{llllll}
    \hline
    Strike & Price & Var & VarG & VarP & VarGP \\
    \hline
     -10 & 10.61 & 112 & 85 & 66 & 48  \\
     0  & 3.66 & 85 & 66 & 33 & 25  \\
     10 & 1.17 & 111 & 52 & 12 & 10 \\
    \hline
  \end{tabular}
  \caption{Basket option in dimension $I=10$ in the Merton model with
  $S_0^i=100$, $r=0.05$, $\sigma^i=0.2$, $\mu^i=0.1$, $\alpha^i=0.3$,
  $\delta^i=0.2$, $\rho = 0.3$, $T=1$, $\omega^i = \frac{1}{I}$ for
  $i=1,\dots,I/2$, $\omega^i = -\frac{1}{I}$ for $i=I/2+1,\dots,I$ and
  $n=50000$. The CPU time for the crude Monte Carlo approach is $0.06$.  The CPU
  times for the importance sampling approach are $(0.17, 0.20, 0.32)$.}
  \label{tab:basket_10d}
\end{table}

The experiments on the one dimensional barrier option (see
Table~\ref{tab:barrier_1d} lead to very similar conclusions regarding the
efficiencies of the different approaches. Roughly speaking, the
\textit{Gaussian} approach does not bring any variance reduction but costs $2.5$
times the CPU times of the crude Monte Carlo approach. The \textit{Poisson}  and
\textit{Gaussian+Poisson} importance sampling approaches do provide an
impressive variance reductions for equivalent computational times at least in
the reduced size approach. The improvement of the optimal variance obtained by
the full size approaches does not look enough to counter balance the extra
computational time. Actually, the reduced size approaches show far better
efficiencies.

\begin{table}[h]
  \centering
  \begin{tabular}{llllll}
    \hline
    Strike & Price & Var & VarG & VarP & VarGP \\
    \hline
     -10 & 10.21 & 60 & 41 & 48 & 29  \\
     0  & 3.35 & 30 & 21 & 22 & 13  \\
     10 & 0.68 & 8.3 & 5.9 & 5.2 & 2.8 \\
    \hline
  \end{tabular}
  \caption{Basket option in dimension $I=10$ in the Merton model with
  $S_0^i=100$, $r=0.05$, $\sigma^i=0.2$, $\mu^i=1$, $\alpha^i=0.1$,
  $\delta^i=0.01$, $\rho = 0.3$, $T=1$, $\omega^i = \frac{1}{I}$ for
  $i=1,\dots,I/2$, $\omega^i = -\frac{1}{I}$ for $i=I/2+1,\dots,I$ and
  $n=50000$. The CPU time for the crude Monte Carlo approach is $0.06$.  The CPU
  times for the importance sampling approach are $(0.17, 0.20, 0.32)$.}
  \label{tab:basket_10d-2}
\end{table}

Since basket options are not path dependent derivatives, the full and reduced
size approaches coincide and we do not distinguish between the two in
Tables~\ref{tab:basket_10d} and~\ref{tab:basket_10d-2}. In these tables, we can
see that the \textit{Gaussian+Poisson} approach provides better variance
reductions that the pure \textit{Poisson} approach, which in turn outperforms
the pure \textit{Gaussian} strategy. However, except for out of the money
options, the gain brought by the different importance sampling approaches do not
compensate the extra computational time in order to keep up with the crude Monte
Carlo strategy. This lack of efficiency mainly comes from the very simple form
of the payoff which makes the crude Monte Carlo method very fast.

\paragraph{Multidimensional barrier option.} We consider a discrete monitoring
down and out barrier option on a basket of assets with payoff
\begin{align*}
  \left( \sum_{i=1}^I \omega^i S^i_T - K \right)_+
  \ind{\forall 1 \le i \le I, \; \forall 1 \le j \le J, \; S^i_{t_j} >
  b^i}
\end{align*}
where the vector $b \in \R^I$ is a lower barrier.
\begin{table}[h]
  \centering
  \begin{tabular}{lllllll}
    \hline
    & Strike & Price & Var & VarG & VarP & VarGP \\
    \hline
    Full & 0  & 0.59 & 7.00 & 4.03 & 4.36 & 1.98  \\
    Reduced &   & 0.59 & 7.00 & 3.51 & 4.36 & 2.05  \\
    Full & -5  & 1.06 & 13.33 & 8.43 & 9.64 & 4.79  \\
    Reduced &   & 1.06 & 13.33 & 8.56 & 9.81 & 5.42  \\
    Full & -10  & 1.64 & 24.26 & 16.57 & 18.39 & 10.57  \\
    Reduced &   & 1.64 & 24.26 & 16.77 & 18.96 & 11.68  \\
    \hline
  \end{tabular}
  \caption{Barrier option in dimension $I=10$ in the Merton model with
  $S_0^i=100$, $r=0.05$, $\sigma^i=0.2$, $\mu^i=1$, $\alpha^i=0.1$,
  $\delta^i=0.01$, $b^i = 80$, $\rho = 0.3$, $T=1$, $\omega^i = \frac{1}{I}$ for
  $i=1,\dots,I/2$, $\omega^i = -\frac{1}{I}$ for $i=I/2+1,\dots,I$ and $J=12$,
  $n=50000$. The CPU time for the crude Monte Carlo approach is $0.76$.  The CPU
  times for the reduced importance sampling approach are $(1.42, 1.44, 1.52)$
  and for the full importance sampling approach they are $(1.53, 2.41, 3.05)$.}
  \label{tab:barrier_10d}
\end{table}

\begin{table}[h]
  \centering
  \begin{tabular}{lllllll}
    \hline
    & Strike & Price & Var & VarG & VarP & VarGP \\
    \hline
    Full & 100  & 2.97 & 36 & 37.8 & 16 & 16  \\
    Reduced &   & 2.97 & 36 & 36.2 & 16 & 16  \\
    Full & 90  & 12.52 & 36 & 36.4 & 14.5 & 14.5  \\
    Reduced &   & 12.52 & 36 & 36 & 14.3 & 14.3  \\
    Full & 110  & 1.64 & 12.1 & 13.3 & 6.1 & 5.5  \\
    Reduced &   & 0.80 & 12.1 & 12 & 5.3 & 5.4  \\
    \hline
  \end{tabular}
  \caption{Barrier option in dimension $I=5$ in the BNS model with $S_0^i=100$,
  $r=0.05$, $\vl^i=0.01$, $\mu^i=1$, $\kappa^i=0.54$, $\beta^i=18.6$, $b^i =
  70$, $\rho = 0.2$, $T=1$, $\omega^i = \frac{1}{I}$ for $i=1,\dots,I$ and
  $J=12$, $n=50000$. The CPU time for the crude Monte Carlo approach is $0.52$.
  The CPU times for the reduced importance sampling approach are $(1.06, 1.1,
  1.17)$ and for the full importance sampling approach they are $(2.1, 3.8,
  10.5)$.}
  \label{tab:barrier_5d}
\end{table}

Our last two examples deal multi-dimensional barrier options with discrete
monitoring. The first striking result to notice when looking at
Tables~\ref{tab:barrier_10d} and~\ref{tab:barrier_5d} concerns the huge CPU
times of the full approaches which nonetheless do not significantly improve the
variance reduction compared to the reduced size methods. This remark definitely
advocates the use of reduced size approaches.  The variance is always divided by
a factor between $2$ and $3$, whereas the CPU time is only twice the one of the
crude Monte Carlo approach. In the Merton model case
(Table~\ref{tab:barrier_10d}), the \textit{Gaussian+Poisson} approach always
provides the largest variance reduction for a computational time very close to
the two other methods, meanwhile in the BNS model (Table~\ref{tab:barrier_5d})
the \textit{Poisson} and \textit{Gaussian+Poisson} methods perform similarly.
The efficiency of the pure \textit{Poisson} approach comes from the particular
form of the BNS model which includes jumps in the volatility process. These
jumps seem to have a larger impact on the overall variance that the Brownian
motion itself.

\section{Conclusion}

In this work, we have studied an importance sampling based Monte Carlo method for jump
processes. The proposed algorithm splits into two parts. First, we compute the
optimal change of measure using Newton's algorithm on a sample average
approximation of the stochastic optimization problem. This method is very robust
and does not require any fine tuning unlike stochastic approximation methods.
Second, we use this estimator of the optimal measure change in an independent
Monte Carlo method. We have established a several convergence results for this
approach and in particular we have proved that it satisfies a Central Limit
Theorem with optimal limiting variance. All the numerical examples we have
investigated advocates the use of reduced size problems to significantly speed up the
computations since the loss of variance reduction compared to the full approach
remains negligible. This importance sampling approach proves to be all the more
efficient as the jump and the diffusion parts are mixed up.

\bibliographystyle{abbrvnat}
\bibliography{biblio}

\begin{thebibliography}{19}
\providecommand{\natexlab}[1]{#1}
\providecommand{\url}[1]{\texttt{#1}}
\expandafter\ifx\csname urlstyle\endcsname\relax
  \providecommand{\doi}[1]{doi: #1}\else
  \providecommand{\doi}{doi: \begingroup \urlstyle{rm}\Url}\fi

\bibitem[Arouna(2004)]{arouna2}
B.~Arouna.
\newblock Adaptative {M}onte {C}arlo method, a variance reduction technique.
\newblock \emph{Monte Carlo Methods Appl.}, 10\penalty0 (1):\penalty0 1--24,
  2004.

\bibitem[Arouna(Winter 2003/04)]{arouna1}
B.~Arouna.
\newblock Robbins {M}onro algorithms and variance reduction in finance.
\newblock \emph{J. of Computational Finance}, 7\penalty0 (2), Winter 2003/04.

\bibitem[Barndorff-Nielsen and Shephard(2001{\natexlab{a}})]{bns01-1}
O.~Barndorff-Nielsen and N.~Shephard.
\newblock Modelling by lévy processess for financial econometrics.
\newblock In O.~Barndorff-Nielsen, S.~Resnick, and T.~Mikosch, editors,
  \emph{Lévy Processes}, pages 283--318. Birkhäuser Boston,
  2001{\natexlab{a}}.

\bibitem[Barndorff-Nielsen and Shephard(2001{\natexlab{b}})]{bns01-2}
O.~E. Barndorff-Nielsen and N.~Shephard.
\newblock Non-gaussian ornstein–uhlenbeck-based models and some of their uses
  in financial economics.
\newblock \emph{Journal of the Royal Statistical Society: Series B (Statistical
  Methodology)}, 63\penalty0 (2):\penalty0 167--241, 2001{\natexlab{b}}.

\bibitem[Barndorff-Nielsen and Stelzer(2013)]{BNS12}
O.~E. Barndorff-Nielsen and R.~Stelzer.
\newblock The multivariate supou stochastic volatility model.
\newblock \emph{Mathematical Finance}, 23\penalty0 (2):\penalty0 275--296,
  2013.
\newblock \doi{10.1111/j.1467-9965.2011.00494.x}.

\bibitem[Carr et~al.(1999)Carr, Madan, and Smith]{CarrMadan99}
P.~Carr, D.~B. Madan, and R.~H. Smith.
\newblock Option valuation using the fast fourier transform.
\newblock \emph{Journal of Computational Finance}, 2:\penalty0 61--73, 1999.

\bibitem[Chen and Zhu(1986)]{chen86}
H.~Chen and Y.~Zhu.
\newblock Stochastic approximation procedure with randomly varying truncations.
\newblock \emph{Scientia Sinica}, 29\penalty0 (9):\penalty0 914--926, 1986.

\bibitem[Jourdain and Lelong(2009)]{jour-lelo-09}
B.~Jourdain and J.~Lelong.
\newblock Robust {A}daptive {I}mportance {S}ampling for {N}ormal {R}andom
  {V}ectors.
\newblock \emph{Ann. Appl. Probab.}, 19\penalty0 (5):\penalty0 1687--1718,
  2009.

\bibitem[Kawai(2007)]{kawai07}
R.~Kawai.
\newblock Adaptive {M}onte {C}arlo variance reduction with two-time-scale
  stochastic approximation.
\newblock \emph{Monte Carlo Methods Appl.}, 13\penalty0 (3):\penalty0 197--217,
  2007.

\bibitem[Kawai(2008{\natexlab{a}})]{Kawai-levy08}
R.~Kawai.
\newblock Adaptive {M}onte {C}arlo variance reduction for {L}\'evy processes
  with two-time-scale stochastic approximation.
\newblock \emph{Methodol. Comput. Appl. Probab.}, 10\penalty0 (2):\penalty0
  199--223, 2008{\natexlab{a}}.

\bibitem[Kawai(2008{\natexlab{b}})]{Kawai-levy09}
R.~Kawai.
\newblock Optimal importance sampling parameter search for {L}\'evy processes
  via stochastic approximation.
\newblock \emph{SIAM J. Numer. Anal.}, 47\penalty0 (1):\penalty0 293--307,
  2008{\natexlab{b}}.

\bibitem[Kiessling and Tempone(2011)]{Kiessling:2011uq}
J.~Kiessling and R.~Tempone.
\newblock Diffusion approximation of {L}\'evy processes with a view towards
  finance.
\newblock \emph{Monte Carlo Methods Appl.}, 17\penalty0 (1):\penalty0 11--45,
  2011.

\bibitem[Kou(2002)]{Kkou02}
S.~G. Kou.
\newblock A jump-diffusion model for option pricing.
\newblock \emph{Manage. Sci.}, 48\penalty0 (8):\penalty0 1086--1101, 2002.

\bibitem[Lapeyre and Lelong(2011)]{lap-lelo-11}
B.~Lapeyre and J.~Lelong.
\newblock A framework for adaptive {Monte}--{Carlo} procedures.
\newblock \emph{Monte Carlo Methods Appl.}, 17\penalty0 (1), 2011.

\bibitem[Ledoux and Talagrand(1991)]{ledouxtalagrand}
M.~Ledoux and M.~Talagrand.
\newblock \emph{Probability in {B}anach spaces}, volume~23 of \emph{Ergebnisse
  der Mathematik und ihrer Grenzgebiete (3) [Results in Mathematics and Related
  Areas (3)]}.
\newblock Springer-Verlag, Berlin, 1991.

\bibitem[Lelong(2008)]{lelong_as}
J.~Lelong.
\newblock Almost sure convergence of randomly truncated stochastic algorithms
  under verifiable conditions.
\newblock \emph{Statistics \& Probability Letters}, 78\penalty0 (16), 2008.

\bibitem[Lelong(2011)]{lel:tcl:11}
J.~Lelong.
\newblock Asymptotic normality of randomly truncated stochastic algorithms.
\newblock \emph{ESAIM. Probability and Statistics}, 2011.

\bibitem[Lemaire and Pag{\`e}s(2010)]{LP10}
V.~Lemaire and G.~Pag{\`e}s.
\newblock Unconstrained {R}ecursive {I}mportance {S}ampling.
\newblock \emph{Ann. Appl. Probab.}, 20\penalty0 (3):\penalty0 1029--1067,
  2010.

\bibitem[Rubinstein and Shapiro(1993)]{MR1241645}
R.~Y. Rubinstein and A.~Shapiro.
\newblock \emph{Discrete event systems}.
\newblock Wiley Series in Probability and Mathematical Statistics: Probability
  and Mathematical Statistics. John Wiley \& Sons Ltd., Chichester, 1993.
\newblock Sensitivity analysis and stochastic optimization by the score
  function method.

\end{thebibliography}



\end{document}